\documentclass[a4paper,11pt,reqno]{amsart}
\usepackage[utf8]{inputenc}
\usepackage{amsmath,amssymb,amsfonts}
\usepackage{epsfig}
\usepackage{mathrsfs}
\usepackage{graphicx,color}
\usepackage{eucal}
\usepackage{bbm}

\newtheorem{thm}{Theorem}[section]

\newtheorem{rmk}{Remark}[section]
\newtheorem{lma}{Lemma}[section]
\newtheorem{assu}{Assumption}[section]

\newcommand{\norm}[1]{\left| \! \left| #1\right| \!\right|}
\newcommand{\bnorm}[1]{\big| \! \big| #1\big| \!\big|}
\newcommand{\Xs}{\mathcal{X}}
\newcommand{\Ys}{\mathcal{Y}}

\newcommand{\ip}[1]{\left< #1 \right>}
\newcommand{\bip}[1]{\big< #1 \big>}
\newcommand{\dom}{\mathscr{D}}

\newcommand{\sm}[1]{\left[ \begin{smallmatrix} #1 \end{smallmatrix} \right]}
\newcommand{\bm}[1]{ \begin{bmatrix} #1 \end{bmatrix}}

\numberwithin{equation}{section}
\font\eka=cmex10
\usepackage{ae}
\def\ind{\mathrel{\hbox{\rlap{%
\hbox to 7.5pt{\hrulefill}}\raise6.6pt\hbox{\eka\char'167}}}}
\begin{document}
\title[Iterative state and parameter estimation]{Iterative observer-based state and parameter estimation for linear systems*\footnote{*M\lowercase{anuscript submitted for publication.}}}


\author[Atte Aalto]{Atte Aalto \\ \\ I\lowercase{nria}, U\lowercase{niversit\'e} P\lowercase{aris}--S\lowercase{aclay}, P\lowercase{alaiseau}, F\lowercase{rance}; M$\Xi$DISIM \lowercase{team}}

\thanks{Address: 1 rue Honor\'e d'Estienne d'Orves; 91120 Palaiseau, FR. \\ \indent Email: atte.ej.aalto@gmail.com}

\begin{abstract}
We propose an iterative method for joint state and parameter estimation using measurements on a time interval $[0,T]$ for systems that are backward output stabilizable. Since this time interval is fixed, errors in initial state may have a big impact on the parameter estimate. We propose to use the back and forth nudging (BFN) method for estimating the system's initial state and a Gauss--Newton step between BFN iterations  for estimating the system parameters. Taking advantage of results on the optimality of the BFN method, we show that for systems with skew-adjoint generators, the initial state and parameter estimate minimizing an output error cost functional is an attractive fixed point for the proposed method. We treat both linear source estimation and bilinear parameter estimation problems.

\medskip

\noindent
{\it Keywords: Parameter estimation, system identification, back and forth nudging, output error minimization}

\smallskip
\noindent
{\it 2010 AMS subject classification: 93B30, 35R30, 93C05} 
\end{abstract}

\maketitle


\section{Introduction}

In this paper, we present a method for estimating system parameters from noisy measurements on a given time interval $[0,T]$. For this purpose, we develop an iterative modification of the sequential joint state and parameter estimation method proposed in \cite{Moireau08}, based on using a Luenberger observer for estimating the state trajectory and a (extended) Kalman filter-based estimator for the system parameters. As we now consider estimation over a fixed time interval, any errors in the initial state of the system may be crucial. Therefore, we use the so called \emph{back and forth nudging} (BFN) method proposed in \cite{AB05,AB08} for initial state estimation, alongside with a Gauss--Newton step between each iteration to estimate the parameter. This is motivated by the fact that the extended Kalman filter is equivalent to a Gauss-Newton optimization step on a suitably chosen cost function. Using the results of \cite{OEBFN} on the optimality of the BFN method, in case the system dynamics are governed by a skew-adjoint main operator, the proposed strategy can be regarded as a hybrid optimization method with separate optimization schemes for the initial state and the system parameters. The method can be used also in the case the main operator is not skew-adjoint as long as the BFN method is stable. However, the optimality results do not hold in that case.

We treat both linear source estimation problems and bilinear parameter estimation problems. Let us introduce the linear source estimation case now to be able to better explain the results of the article. The bilinear case will be presented and treated in Section~\ref{sec:bilin}. In the source estimation problems we assume that the system dynamics are given by
\begin{equation} \label{eq:lin_syst}
\begin{cases}
\dot z=Az+B\theta+\eta, \\
y=Cz+\nu, \\
z(0)=z_0.
\end{cases}
\end{equation}
Here $\eta$ represents the unknown modeling error and input noise and $\nu$ represents output noise and modeling errors related to the measurement. Now $\theta \in \Theta$ is the parameter that we are interested in, and  $B(\cdot) \in L^2(0,T; \mathcal{L}(\Theta,\Xs))$. The state space $\Xs$ and the parameter space $\Theta$ are assumed to be separable Hilbert spaces.

The idea is to estimate the initial state $z_0$ and the parameter $\theta$ of system \eqref{eq:lin_syst} by minimizing the (regularized) output error discrepancy cost function, defined by
\begin{equation} \label{eq:cost_main}
J(\xi,\zeta):=\ip{\xi-\theta_0,U_0(\xi-\theta_0)}+\int_0^T \norm{y-C\hat z[\xi,\zeta]}_{\Ys}^2d\tau
\end{equation}
where
\begin{equation} \label{eq:obs_main}
\begin{cases}
\dot{\hat z}[\xi,\zeta]=A\hat z[\xi,\zeta]+B\xi+\kappa C^*(y-C\hat z[\xi,\zeta]), \\
\hat z[\xi,\zeta](0)=\zeta.
\end{cases}
\end{equation}
The parameter $\theta_0$ represents our prior knowledge of the parameter and the self-adjoint and positive operator $U_0$ is chosen based on our confidence on the prior. The inner product term can also be interpreted as a Tikhonov regularization term.  We shall treat the minimization problem both with the feedback term in the dynamics of $\hat z$ (that is, with $\kappa>0$), and without the feedback term ($\kappa=0$). We remark that in case the model is erroneus, the minimization problem with $\kappa >0$ may produce better parameter estimates than the traditional output error minimization approach (that is, $\kappa=0$). This is demonstrated by a simple example in Section~\ref{sec:fb}. For the bilinear problem, the system and the cost function are defined in Section~\ref{sec:bilin}.

The theoretical results of this paper are concerned with systems with skew-adjoint generators $A$ and bounded observation operators $C$, under the rather strong exact observability assumption, but the developed method can be used in more general situations. In the main results, we shall show that the minimizer of the output error cost function \eqref{eq:cost_main} is an attractive fixed point for the presented method. In the linear parameter estimation case this can be shown with any noise processes. In addition, the method's domain of convergence is infinite. In the bilinear parameter estimation problem, there are some restrictions on the noise processes and the method's domain of convergence is finite.

One topic that is not treated in this article is the identifiability of the parameter. The reason for this exclusion is that it is difficult to state any results that would hold in a wide variety of scenarios. Consider the "extreme" case that there exists a parameter $\tilde\theta \in \Theta$ for which $B(t)\tilde\theta=0$ for all $t \in [0,T]$. Then obviously no information on the corresponding component of the parameter will ever be obtained from the output $y$ and so because of the regularization term in the cost function \eqref{eq:cost_main}, the minimizer of $J$, denoted by $\theta^o$, satisfies $\bip{\theta^o,\tilde\theta}=\bip{\theta_0,\tilde\theta}$. For results on the identifiability of parameters and persistance of excitation of systems, we refer to \cite{Baumeister,Levanony,Shimkin}.

As mentioned, the proposed method can be regarded as an iterative, finite time horizon modification of the joint state and parameter estimation method presented in \cite{Moireau08}. Other variants of the method are treated in \cite{MoireauHinf} and \cite{Hinf_old} where an $H^{\infty}$ criterion is minimized, and \cite{ROUKF}  where an Unscented Kalman filter (UKF) based strategy is used for the parameter estimation. An application of this method on a cardiac model has been presented in \cite{cardiacUKF} by Marchesseau \emph{et al}. Let us also refer to \cite{Dual_KF} by Mariani and Corigliano for a joint state and parameter estimator utilizing two separate but connected Kalman estimators for the state and parameter. It should be noted that the UKF strategy could also be used together with a BFN approach for initial state estimation. At least the results on the linear case in this paper would hold also for this approach. In addition, it is somewhat easier to implement since one does not need to compute the parameter-state sensitivity operator analytically (cf. EKF vs. UKF). All these methods are also closely related to the classical strategy of using the Extended Kalman Filter on the augmented state vector that contains the actual system state and the parameter vector. Convergence analysis for this strategy is presented in \cite{Ljung_EKF} by Ljung. However, if the system dimension is high, the Kalman-based strategy becomes very costly. Indeed, the EKF requires manipulations of the matrix Riccati equation that has the size of the sum of the state space and parameter space dimensions. In contrast, implementation of the Luenberger observer does not really increase the computational cost of the numerical model. An alternative strategy is based on Monte Carlo approach. A joint state-parameter estimator based on particle filtering is proposed in \cite{ParticleSHM} by Chatzi and Smyth.
See also \cite{Erazo} by Erazo for a review on numerous references and experimental results on different Bayesian approaches for state/parameter estimation in structural health monitoring. 





\section{Background and preliminary results}

\subsection{Observers and the BFN method}



Let us first recall some features of Luenberger observers (\cite{Luenberger},\cite[Section~5.3]{CZ}). The basic idea is to correct the observer state dynamics using a correction term that depends linearly on the output discrepancy. That is, assume that the system dynamics are given by
\begin{equation} \label{eq:system}
\begin{cases}
\dot z = Az+f+\eta, \\
z(0)=z_0, \\
y=Cz+\nu
\end{cases}
\end{equation}
where $f$ is a known load term and $\eta$ represents input noise and modeling errors. The system output is $y$ and $\nu$ is the output noise.  The Luenberger observer dynamics are then given by
\begin{equation} \label{eq:observer}
\begin{cases}
\dot{\hat z}=A\hat z+f+K(y-C\hat z), \\
\hat z(0)=\zeta,
\end{cases}
\end{equation} 
where $K \in \mathcal{L}(\Ys,\Xs)$ is the observer feedback operator and $\zeta$ is our initial state estimate. By superposition, the estimation error $\varepsilon:=z-\hat z$ satisfies
\[
\dot\varepsilon=(A-KC)\varepsilon+\eta+K \nu, \qquad \varepsilon(0)=z_0-\zeta.
\]
The challenge in the observer design is to find the feedback operator such that the closed loop system with the semigroup generator $A-KC$ is stable.

We shall formulate the joint state and parameter estimation method with a general feedback operator $K$, but in the analysis of this paper, we treat only colocated feedbacks $K=\kappa C^*$ with bounded observation operators $C$, where $\kappa>0$ is called the observer gain. For background on observers with colocated feedback, we refer to \cite{Liu} studying systems with skew-adjoint generators and bounded observation operators, and \cite{Curtain_Weiss} studying systems with essentially skew-adjoint and dissipative (ESAD) generators and also unbounded observation operators. For studies on colocated feedback on elastic systems, we refer to \cite{collocated_book} and \cite{Chapelle_wave}.

An assumption that is often needed in results on observers is that the system is \emph{exactly observable at time $T$}, namely that there exists $\gamma >0$ such that
\[
\int_0^T \norm{Ce^{At}\zeta}_{\Ys}^2 dt \ge \gamma^2 \norm{\zeta}_{\Xs}^2
\]
for all $\zeta \in \Xs$. A classical result says that for skew-adjoint generators $A$ and bounded observation operators $C$, the closed loop operator with colocated feedback, $A-\kappa C^*C$ for $\kappa>0$ generates an exponentially stable semigroup if and only if the system is exactly observable at some time $T$ \cite[Theorem~2.3]{Liu}. See also \cite{Haraux} and \cite[Chapter~2]{collocated_book} for related results for second order systems with  bounded and unbounded observation operators, respectively. The reverse direction holds also for dissipative generators. However, if $\kappa$ is too large, the observer becomes over-damped and its performance will deteriorate. As an illuminating example, see \cite{Cox} for a study on energy dissipation for damped wave equation. A feasible gain value $\kappa$ can be determined numerically by spectral methods, as in \cite[Section~3]{Moireau08} or \cite{Chapelle_wave}.

In this article we wish to estimate the initial state of the system efficiently. For some systems --- in particular, for systems with skew-adjoint generators --- this can be done using a Luenberger observer alternately forward and backward in time. This strategy is generally known as \emph{back and forth nudging} (BFN), and it was originally proposed by Auroux and Blum in \cite{AB05} and \cite{AB08}. A more rigorous treatment was carried out by Ramdani \emph{et al.} in \cite{Ramdani10} and further studies include \cite{Haine} by Haine, \cite{HR12} by Haine and Ramdani, \cite{Fridman} by Fridman,  and \cite{OEBFN} by Aalto. In \cite{OEBFN} it is shown that the initial state estimate given by the BFN method converges to the minimizer of the $L^2$-norm of the output discrepancy if the observer gains are taken to zero with a suitable rate. 

In the BFN method, the dynamics of the forward observer in the $j^{\textup{th}}$ iteration, $\hat z_j^+$, are given by the normal observer equation \eqref{eq:observer} (with $K=K_j^+$) initialized from the final state of the backward observer on the previous iteration, that is, $\hat z_j^+(0)=\hat z_{j-1}^-(T)$. The dynamics of the backward observer are given by
\[
\begin{cases}
\dot{\hat z}_j^-(t)=-A\hat z_j^-(t)-f(T-t)+K_j^-(y(T-t)-C\hat z_j^-(t)), \\
\hat z_j^-(0)=\hat z_j^+(T).
\end{cases}
\]
With such definition, the backward estimate $\hat z_j^-(t)$ is an estimate for $z(T-t)$ and so the initial state estimate after $j$ iterations is given by $\hat z_j^-(T)$. The main result of \cite{Ramdani10} is that if there are no noises, that is, $\eta,\nu=0$ in \eqref{eq:system}, and if the system is both forward and backward stabilizable, that is, there exist forward and backward feedback operators $K^{\pm}$ such that $A-K^+C$ and $-A-K^-C$ are exponentially stable, then the initial state estimate given by the BFN method converges exponentially to the true initial state of \eqref{eq:system}. The main result of \cite{OEBFN} is concerned with colocated feedback $K_j^{\pm}=\kappa_j C^*$, and it states that if $A$ is skew-adjoint and the system is exactly observable at time $T$, then assuming that the observer gains satisfy $\sum_{j=1}^{\infty} \kappa_j=\infty$ and $\sum_{j=1}^{\infty} \kappa_j^2 < \infty$, then the initial state estimate converges to the minimizer of the cost function
\[
J_0(x):=\int_0^T \norm{y(t)-C\hat z[x](t)}_{\Ys}^2 dt
\]
where $\hat z[x]$ is the solution to
\[
\begin{cases}
\dot{\hat z}[x]=A\hat z[x]+f, \\
\hat z[x](0)=x.
\end{cases}
\]
This is an important result from the point of view of this article, since we wish to use the BFN method as a  minimization scheme. We do not use the results of \cite{OEBFN} directly, but the techniques used in this paper are similar to the techniques used there.

\subsection{Preliminary results}

The main result of this section is that in the linear case the cost function $J$ given in \eqref{eq:cost_main}, is strictly convex under some assumptions. In the bilinear case, a small modification of these results show that under the exact observability assumption, having the Tikhonov regularization term in the cost function \eqref{eq:cost_bilin} only for the parameter $\xi$ is sufficient to make the curvature of the attainable set sufficiently small in the sense of \cite{Chavent}. This means that with sufficiently big regularization, the optimization problem has a unique solution if the parameter set is sufficiently restricted. 

First, we need an auxiliary result showing that if the system is exactly observable, then also the closed loop observer system is exactly observable, which we will show for general, bounded feedback operators.
\begin{lma} \label{lma:obs_pres}
Assume  $\norm{Ce^{At}x}_{L^2(0,T;\Ys)} \ge \gamma_0 \norm{x}_{\Xs}$ and that the semigroup $e^{At}$ is contractive.  Then it holds that $\norm{Ce^{(A-KC)t}x}_{L^2(0,T;\Ys)} \ge \gamma \norm{x}_{\Xs}$ with $\gamma=\frac{\gamma_0\sqrt{2}}{\sqrt{2}+T\norm{C}\norm{K}}$.
\end{lma}
\begin{proof}
Fix $x \in \Xs$ and consider the difference of the trajectories
\[
e^{At}x-e^{(A-KC)t}x=\int_0^t e^{A(t-s)}KCe^{(A-KC)s}x \, ds
\]
given by the semigroup perturbation formula (see \cite[Section~3.1]{Pazy}). As $\norm{e^{At}} \le 1$, application of the Cauchy--Schwartz inequality to the integral yields
\[
\bnorm{Ce^{At}x-Ce^{(A-KC)t}x}_{\Ys} \le \norm{C}\norm{K}\sqrt{t}\bnorm{Ce^{(A-KC)s}x}_{L^2(0,t;\Ys)}.
\]
Using this pointwise estimate together with the observability assumption and the triangle inequality in $L^2(0,T;\Ys)$ gives
\begin{align*}
\bnorm{Ce^{(A-KC)t}x}_{L^2(0,T;\Ys)} & \ge \gamma_0\norm{x}_{\Xs}-\norm{C}\norm{K}\left(\int_0^T \! t \bnorm{Ce^{(A-KC)s}x}_{L^2(0,t;\Ys)}^2 dt \right)^{1/2} \\ & \ge \gamma_0 \norm{x}_{\Xs}-\frac{\norm{C}\norm{K}T}{\sqrt{2}}\bnorm{Ce^{(A-KC)t}x}_{L^2(0,T;\Ys)}
\end{align*}
where the second inequality follows again from the Cauchy--Schwartz inequality.
The last term is then moved to the left hand side, after which the claim follows by multiplying both sides by $\frac{\sqrt{2}}{\sqrt{2}+T\norm{C}\norm{K}}$.
\end{proof}

The main result of this section is the strict convexity of the cost function $J$ in the linear case, implying the existence of a unique minimizer.
\begin{lma} \label{lma:convex}
Assume that the system is exactly observable, $A$ generates a contractive semigroup, and $U_0 \ge \delta I$ for some $\delta > 0$. Then the cost function $J$ defined in \eqref{eq:cost_main} is strictly convex.
\end{lma}
\begin{proof}
The cost function $J$ can be rewritten to
\[
J(\xi,\zeta)=\norm{\bm{y \\ \sqrt{U_0}\theta_0}-\bm{C\hat z[\xi,\zeta] \\ \sqrt{U_0}\xi}}_{L^2(0,T;\Ys) \times \Theta}^2
\]
where the square root operator $\sqrt{U_0}$ exists since $U_0$ is self-adjoint and strictly positive.
The second term depends on $(\xi,\zeta)$ through an affine map
\[
\bm{C\hat z[\xi,\zeta] \\ \sqrt{U_0}\xi}=\Gamma \bm{\xi \\ \zeta}+b.
\]
The load term $b$ is the contribution of the $\kappa C^*y$ term in \eqref{eq:obs_main} but it does not matter here. 
The coefficient operator is
\[
\Gamma=\bm{Ce^{(A-\kappa C^*C)\cdot} & C\Pi \\ 0 & \sqrt{U_0}} \ : \   \Xs \times \Theta \to L^2(0,T;\Ys) \times \Theta
\]
where $\Pi$ is the sensitivity operator from the parameter to the state estimate, and it is the solution to $\dot\Pi=(A-\kappa C^*C)\Pi+B$, $\Pi(0)=0$. Here it should be understood as an operator mapping from $\Theta$ to $L^2(0,T;\Xs)$.
Since the problem is linear-quadratic, the strict convexity follows from $\Gamma^*\Gamma \ge \epsilon I$ for some $\epsilon > 0$ that will be shown now.

Just for this proof we define a  shorthand notation $\Gamma=\sm{S & P \\ 0 & \sqrt{U}}$ ($S$ and $P$ standing for ``state" and ``parameter") and
then we can write a block form for the product
\[
\Gamma^*\Gamma=\bm{ S^*S & S^*P \\ P^*S & P^*P+U}.
\]
We intend to show that there exists $\epsilon>0$ such that  $\Gamma^*\Gamma-\lambda I \ge 0$ for $\lambda \in [0,\epsilon]$.
The rest of the proof is based on the theory of Schur complements in Hilbert space (see, e.g., \cite{Schur}, Thm.~2.2 in particular). It holds that $\Gamma^*\Gamma -\lambda I \ge 0$ if and only if $S^*S -\lambda I \ge 0$, $P^*P+U-\lambda I \ge 0$, and
\[
S^*S-S^*P\big(P^*P+U-\lambda I\big)^{-1}P^*S-\lambda I \ge 0.
\] 
By Lemma~\ref{lma:obs_pres} and the exact observability assumption, it holds that $S^*S \ge \gamma^2I$ where $\gamma>0$ is given in Lemma~\ref{lma:obs_pres}, and $U \ge \delta I$ was assumed in the theorem. Therefore the first two conditions are satisfied if we assume $\lambda < \min( \gamma^2,\delta)$. We shall then find an upper bound for $\lambda$ so that the last condition is satisfied. 
It holds that
\begin{align*}
& S^*S-S^*P\big(P^*P+U-\lambda I\big)^{-1}P^*S-\lambda I \\
& \ge S^*S-S^*P\big(P^*P+(\delta-\lambda) I\big)^{-1}P^*S-\lambda I \\
&=S^*\left(I+\frac1{\delta-\lambda}PP^* \right)^{-1}S-\lambda I
\end{align*}
where the last equality holds by the Woodbury identity. For the term inside parentheses, it holds that $I+\frac1{\delta-\lambda}PP^*  \le \frac{\norm{P}^2+\delta-\lambda}{\delta-\lambda}I$, and hence
\[
S^*\left(I+\frac1{\delta-\lambda}PP^* \right)^{-1}S-\lambda I \ge \frac{\delta-\lambda}{\norm{P}^2+\delta-\lambda}S^*S-\lambda I 
\]
which is nonnegative if (for example) $\lambda \le \min \left( \frac{\delta}2 , \frac{\gamma^2 \delta}{2(\norm{P}^2+\delta)} \right)$. Thus it holds that $\Gamma^*\Gamma \ge \epsilon I$ where $\epsilon=\frac12 \min \left(\delta , \frac{\gamma^2 \delta}{\norm{P}^2+\delta} \right)>0$.
\end{proof}

\section{Linear parameter estimation}

Let us start with the source identification problem that was already formulated in \eqref{eq:lin_syst}, \eqref{eq:cost_main}, and \eqref{eq:obs_main} in the Introduction.
As mentioned, the parameter estimator in \cite{Moireau08} is based on the (extended) Kalman filter. If the method is applied on the time interval $[0,T]$, in the linear case the Kalman filter based strategy actually corresponds to a Gauss--Newton optimization step. Motivated by this observation, we propose a hybrid method for estimating the system's initial state and parameters. In this method, the initial state is estimated by the BFN method, and the parameters by taking one Gauss--Newton step between the BFN iterations. As in \cite{OEBFN}, we intend to use variable gain in the BFN iterations (in general). Therefore also the cost function to which the Gauss--Newton method is applied, changes between iterations. Namely, at every iteration, we take one Gauss--Newton step applied to the cost function
\[
J_j(\xi):=\ip{\xi-\theta_0,U_0(\xi-\theta_0)}+\int_0^T \norm{y(t)-C\hat z_j[\xi](t)}_{\Ys}^2 dt,
\]
where $\hat z_j[\xi]$ is the solution to
\[
\begin{cases}
\dot{\hat z}_j[\xi]=A \hat z_j[\xi]+B\xi+K_j(y-C\hat z_j[\xi]), \\
\hat z_j[\xi](0)=\hat z_{j-1}^-(T).
\end{cases}
\]

For the Gauss--Newton method, we need to compute the derivative of the estimated output with respect to the parameter. This can be obtained through the sensitivity operator $\Pi_j$, which is defined as the (Fr\'echet) derivative of the state estimate with respect to the parameter. The forward part of the state estimation together with the Gauss--Newton step is
\begin{equation} \label{eq:linear_joint}
 \begin{cases}
 \dot{\hat z}_j^+=A\hat z_j^+ +K_j(y-C\hat z_j^+)+B\hat\theta_j, \quad \hat z_j^+(0)=\hat z_{j-1}^-(T), \\
 \dot{\Pi}_j= (A-K_jC)\Pi_j + B, \quad \Pi_j(0)=0,     \\
 \dot{U}= \Pi_j^*C^*C\Pi_j, \quad U(0)=U_0, \\
 \dot{\xi}= \Pi_j^{*}C^*\big(y-C\hat z_j^+ \big), \quad \xi(0)=0, \\
 \hat\theta_{j+1}=\hat\theta_j-U(T)^{-1}\big(U_0(\hat\theta_j-\theta_0)-\xi(T)\big).
  \end{cases}
 \end{equation}
 The backward state observer is defined "forward in time" by
 \begin{equation} \label{eq:bwd_lin}
 \begin{cases}
 \dot{\hat z}_j^-(t)=-A\hat z_j^-(t) +K_j(y(T-t)-C\hat z_j^-(t))-B(T-t)\hat\theta_{j+1}, \\ \hat z_j^-(0)=\hat z_j^+(T)+\Pi_j(T)(\hat\theta_{j+1}-\hat\theta_j).
 \end{cases}
 \end{equation} 
 With such definition, $\hat z_j^-(t)$ is actually an estimate for $z(T-t)$. 
 
 The method has been defined for general stabilizing feedback operators $K_j$ but in order to be able to show any optimality results, we shall resort to the case of colocated feedback $K_j=\kappa_j C^*$ in the analysis in the following sections.

 \subsection{Optimization problem with feedback}
 In this section, we consider the method with constant gain, $\kappa_j=\kappa > 0$. The ultimate goal is to show that the state and parameter estimates given by the method \eqref{eq:linear_joint} and \eqref{eq:bwd_lin} converge to the unique solution $\zeta^o \in \Xs,\theta^o \in \Theta$ of the following optimization problem:
\begin{equation} \label{eq:cost}
\min_{\zeta \in \Xs, \xi \in \Theta} \ip{\xi-\theta_0,U_0(\xi-\theta_0)}_{\Theta}+ \int_0^T\norm{y-C\hat z[\zeta,\xi]}_{\Ys}^2 d\tau
\end{equation}
where $\hat z[\zeta,\xi]$ is the solution to
\begin{equation} \label{eq:obs1}
\begin{cases}
\dot{\hat z}[\zeta,\xi]=A\hat z[\zeta,\xi]+B\xi+\kappa C^*(y-C\hat z[\zeta,\xi]), \\
\hat z[\zeta,\xi](0)=\zeta.
\end{cases}
\end{equation}
The existence of a unique solution follows from the strict convexity of the cost function, which is shown in Lemma~\ref{lma:convex}. Denote the optimal trajectory by $z^o:=\hat z[\zeta^o,\theta^o]$ and the corresponding output residual by $\chi:=y-Cz^o$. 

As the dependence of the observer output  $C\hat z$ on the parameter $\xi$ is linear, the Gauss--Newton step actually gives the optimal parameter corresponding to the initial state estimate. Thus, at every iteration, the parameter estimate $\hat\theta_{j+1}$ is the solution to the minimization problem
\begin{equation} \nonumber
\min_{\xi \in \Theta} \ip{\xi-\theta_0,U_0(\xi-\theta_0)}_{\Theta}+ \int_0^T\bnorm{y-C\hat z[\hat z_j^+(0),\xi]}_{\Ys}^2 d\tau
\end{equation}
where $\hat z[\hat z_j^+(0),\xi]$ is defined in \eqref{eq:obs1}. In addition, it is easy to verify by differentiation that for any two parameters $\xi_1,\xi_2 \in \Theta$, it holds that
\[
\hat z[\xi_1,\zeta](t)=\hat z[\xi_2,\zeta](t)+\Pi(t)\big( \xi_1-\xi_2 \big).
\]
Thus, because of the correction made to the initial state of the backward observer in \eqref{eq:bwd_lin}, the new initial state estimate $\hat z_j^-(T)$ does not depend on the previous parameter estimate $\hat\theta_j$. Therefore, to obtain convergence of the method, it suffices to study the sequence of the initial state estimates $\hat z_j^-(T)$. If we can show $\hat z_j^-(T) \to \zeta^o$ as $j \to \infty$, then also $\hat\theta_j \to \theta^o$ as $j \to \infty$.

It is possible to interpret \eqref{eq:linear_joint} and \eqref{eq:bwd_lin} as a mapping $\hat z_j^-(T)=f(\hat z_{j-1}^-(T))$ where $f:\Xs \to \Xs$ is an affine transformation. Next we show that the optimal initial state $\zeta^o$ is a fixed point of this mapping.
\begin{lma} \label{lma:fix_lin}
Assume $A$ is skew-adjoint. Then the minimizer $\zeta^o$ is a fixed point of the mapping $f$, that is, $f(\zeta^o)=\zeta^o$.
\end{lma}
\begin{proof}
As concluded above, in the linear case the new state and parameter estimates $\hat z_j^-(T)$ and $\hat\theta_{j+1}$ do not depend on the old parameter estimate $\hat\theta_j$. Therefore, to simplify computations, we can set $\hat\theta_j=\theta^o$. As the trajectory $\hat z_j^+$ is initialized from the optimal initial state, and also the parameter is the optimal parameter, it holds that $\hat z_j^+=z^o$. Therefore it also holds that $\hat\theta_{j+1}=\theta^o$.
The dynamics of the backward estimation error $\varepsilon^-(t):=z^o(T-t)-\hat z_j^-(t)$ are given by
\begin{align*}
\dot \varepsilon^-(t) &= (-A+\kappa C^*C)z^o(T-t)-(-A-\kappa C^*C)\hat z^-(t)-2\kappa C^*y(T-t) \\ &= (-A-\kappa C^*C)\varepsilon^-(t)-2\kappa C^*\chi(T-t)
\end{align*}  
with $\varepsilon^-(0)=0$. In the second equality, we used $y=Cz^o+\chi$. At time $t=T$, it holds that
\begin{align*}
\varepsilon^-(T)&=-2\kappa \int_0^T e^{(-A-\kappa C^*C)(T-s)}C^*\chi(T-s)ds \\ &=-2\kappa \int_0^Te^{(A^*-\kappa C^*C)s}C^*\chi(s)ds=0
\end{align*}
because the integral is exactly the derivative of the cost function \eqref{eq:cost} with respect to the initial state $\zeta$ at the optimum $\zeta=\zeta^o$. Thus $\hat z^-(T)=\zeta^o$ concluding the proof.
\end{proof}

Next we show that for small enough observer gain $\kappa>0$, there exists $k<1$ such that for any $\Delta\zeta \in \Xs$,
\[
\norm{f(\zeta^o+\Delta\zeta)-\zeta^o}_{\Xs} \le k \norm{\Delta\zeta}_{\Xs}
\]
meaning that $\zeta^o$ is an attractive fixed point, ensuring convergence of the method. At this point, define also the shifted mapping $g(\cdot):=f(\cdot+\zeta^o)-\zeta^o$ that has a fixed point $0$.

\begin{thm} \label{thm:attractive_lin}
Assume that $A$ is skew-adjoint and that the system is exactly observable, that is, there exists $\gamma > 0$ such that $\norm{Ce^{A\cdot}x}_{L^2(0,T;\Ys)} \ge \gamma \norm{x}_{\Xs}$ for all $x \in \Xs$. Assume also $U_0 \ge \delta I > 0$.

Then the function $g:\Xs \to \Xs$ defined through equations \eqref{eq:linear_joint} and \eqref{eq:bwd_lin} satisfies 
\[
\norm{g(\zeta)}_{\Xs} \le \big(1-\alpha\kappa+\mathcal{O}(\kappa^2)\big)\norm{\zeta}_{\Xs}
\]
where 
$\alpha=\min \left( \frac{\delta \gamma^2}{T^2\norm{C}_{\mathcal{L}(\Xs,\Ys)}^2\norm{B}_{L^2(0,T;\mathcal{L}(\Theta,\Xs))}^2},\gamma^2 \right)$.
\end{thm}

\begin{proof}
Firstly note that if $\kappa =0$, it holds that $g(\zeta)=\zeta$ for all $\zeta \in \Xs$. The idea in the proof is to show that $\frac{d}{d\kappa}\norm{g(\zeta)}_{\Xs}^2 \big|_{\kappa=0} \le -2\alpha \norm{\zeta}_{\Xs}^2$ for some $\alpha>0$ that does not depend on $\zeta$.

To get started, fix $\Delta \zeta \in \Xs$. As noted before Lemma~\ref{lma:fix_lin}, in the linear case the old parameter estimate $\hat\theta_j$ does not have any effect on the new state estimate $\hat z_j^-(T)$, because the trajectory is corrected to correspond to the new parameter in the beginning of the backward phase (see \eqref{eq:bwd_lin}). Neither does the old parameter affect the new parameter estimate $\hat\theta_{j+1}$. Hence we can use the parameter estimate $\hat\theta_{j+1}$ (from which we henceforth drop the index $j+1$) given by \eqref{eq:linear_joint} also in the forward phase. Since $f$ and therefore also $g$ are affine mappings in $\Xs$, the load term does not play a role in the convergence analysis, and therefore we can set $\theta_0=0$ and $y=0$. Clearly the optimum is then $\zeta^o=0$ and $\theta^o=0$ implying that also the whole trajectory $z^o$ is in fact zero. Now $g(\Delta\zeta)$ is the end state $\hat z^-(T)$ of the backward observer given by \eqref{eq:bwd_lin}. The proof is divided into two parts. In the first part, we derive a feasible expression for the derivative $\frac{d}{d\kappa} \frac12 \norm{\hat z^-(T)}_{\Xs}^2 \Big|_{\kappa=0}$ and, in the second part, we compute an upper bound for this expression.

{\bf Part 1:}
When the load terms are removed, the dynamics equations for the forward and backward estimates are simply
\[
\begin{cases}
\dot{\hat z}^+=(A-\kappa C^*C)\hat z^++B\hat \theta, \\ \hat z^+(0)=\Delta\zeta,
\end{cases}
\]
and
\[
\begin{cases}
\dot{\hat z}^-(t)=(-A-\kappa C^*C)\hat z^-(t)-B(T-t)\hat\theta, \\
\hat z^-(0)=\hat z^+(T).
\end{cases}
\]
The backward equation can be written in the form
\[
\dot{\hat z}^-(t)=(-A+\kappa C^*C)\hat z^-(t)-B(T-t)\hat\theta -2\kappa C^*C \hat z^-(t).
\]
If the last term here is neglected, then the equation is exactly the time-inverted forward equation, and then $\hat z^-(t)=\hat z^+(T-t)$. In addition, by applying the semigroup perturbation formula, the effect of the last term can be separated, and so it holds that
\begin{equation} \label{eq:confer}
\hat z^-(T)=\Delta\zeta -2\kappa \int_0^T e^{-(A-\kappa C^*C)(T-s)}C^*C\hat z^-(s)ds.
\end{equation} 
Using this and $-A=A^*$, we conclude
\begin{align*}
\frac{d}{d\kappa} \frac12 \norm{\hat z^-(T)}_{\Xs}^2 \Big|_{\kappa=0}&=\ip{\Delta\zeta,\frac{d}{d\kappa} \hat z^-(T)\big|_{\kappa=0}}_{\!\! \Xs} \\
&=-2\int_0^T \ip{\Delta\zeta,e^{-A(T-s)}C^*C\hat z^-(s)}_{\! \Xs}ds \\ & = -2 \int_0^T \ip{ Ce^{A(T-s)}\Delta\zeta,C\hat z^-(s) }_{\! \Ys}ds.
\end{align*}
Changing the integration variable $t=T-s$ and recalling that with $\kappa=0$,
\[
\hat z^-(s)=\hat z^+(T-s)=e^{At}\Delta\zeta+\int_0^t e^{A(t-r)}B(r)\hat\theta dr
\]
finally yields (henceforth we use the shorter notation $\Pi(t)\hat\theta=\int_0^t e^{A(t-r)}B(r)\hat\theta dr$)
\begin{align} \label{eq:ddk}
\frac{d}{d\kappa} \frac12 \norm{\hat z^-(T)}_{\Xs}^2 \Big|_{\kappa=0} =& -2 \norm{Ce^{A \cdot}\Delta\zeta}_{L^2(0,T;\Ys)}^2 \\ \nonumber & -2 \ip{Ce^{A\cdot}\Delta\zeta,C\Pi\hat\theta}_{\! \! L^2(0,T;\Ys)}.
\end{align}

It holds that $\bm{C\Pi \\ I}\hat \theta$ is the orthogonal projection of $\bm{-Ce^{A\cdot}\Delta\zeta \\ 0}$ onto the subspace $\bm{C\Pi \\ I}\Theta \subset L^2(0,T;\Ys) \times \Theta$ with $\Theta$ equipped with norm $\ip{\theta,U_0\theta}$. Thus it holds that 
\begin{align} \label{eq:proj_ip}
&\ip{Ce^{A\cdot}\Delta\zeta,C\Pi\hat\theta}_{\! \! L^2(0,T;\Ys)} =-\bip{\hat\theta,U_0\hat\theta}-\bnorm{C\Pi\hat\theta}_{L^2(0,T;\Ys)}^2.
\end{align}
In addition, due to the orthogonality of the projection, the "Pythagorean law" gives 
\begin{align} \label{eq:pythagoras}
 \norm{Ce^{A\cdot}\Delta\zeta}_{L^2(0,T;\Ys)}^2=& \bip{\hat\theta,U_0\hat\theta}+\bnorm{C\Pi\hat\theta}_{L^2(0,T;\Ys)}^2 \\ \nonumber & +\bip{\hat\theta,U_0\hat\theta}+\bnorm{Ce^{A\cdot}\Delta\zeta+C\Pi\hat\theta}_{L^2(0,T;\Ys)}^2.
 \end{align}
Inserting this and \eqref{eq:proj_ip} to \eqref{eq:ddk} gives
\begin{align} \label{eq:ddk_fin}
\frac{d}{d\kappa} \frac12 \norm{\hat z^-(T)}_{\Xs}^2 \Big|_{\kappa=0}=-2\bip{\hat\theta,U_0\hat\theta}-2\bnorm{Ce^{A\cdot}\Delta\zeta+C\Pi\hat\theta}_{L^2(0,T;\Ys)}^2.
\end{align}

{\bf Part 2:} What is left is to find an upper bound for the right hand side of \eqref{eq:ddk_fin} in terms of $\Delta\zeta$. It holds (for $\kappa=0$) that $\Pi(t)\hat\theta=\int_0^t e^{A(t-s)}B(s)\hat\theta ds$ and so --- recalling that $A$ is skew-adjoint and hence $\norm{e^{At}}_{\mathcal{L}(\Xs)}=1$ --- Young's inequality for convolutions gives 
\[
\bnorm{C\Pi\hat\theta}_{L^2(0,T;\Ys)} \le T\norm{C}_{\mathcal{L}(\Xs,\Ys)}\norm{B}_{L^2(0,T;\mathcal{L}(\Theta,\Xs))}\bnorm{\hat\theta}_{\Theta}.
\]
 Using first $U_0 \ge \delta I$ and the inequality $a^2+b^2 \ge \frac12 (a+b)^2$, and then the bound for $\bnorm{C\Pi\hat\theta}_{L^2(0,T;\Ys)}$ gives
\begin{align*}
& 2\bip{\hat\theta,U_0\hat\theta}+2\bnorm{Ce^{A\cdot}\Delta\zeta+C\Pi\hat\theta}_{L^2(0,T;\Ys)}^2 \\ & \ge  \left( \sqrt{\delta}\bnorm{\hat\theta}+\bnorm{Ce^{A\cdot}\Delta\zeta+C\Pi\hat\theta}_{L^2(0,T;\Ys)} \right)^2 \\ & \ge \left( \frac{\sqrt{\delta} \, \bnorm{C\Pi\hat\theta}_{L^2(0,T;\Ys)}}{T\norm{C}_{\mathcal{L}(\Xs,\Ys)}\norm{B}_{L^2(0,T;\mathcal{L}(\Theta,\Xs))}}+\bnorm{Ce^{A\cdot}\Delta\zeta+C\Pi\hat\theta}_{L^2(0,T;\Ys)} \right)^2  \\ & \ge M \left( \bnorm{C\Pi\hat\theta}_{L^2(0,T;\Ys)}+\bnorm{Ce^{A\cdot}\Delta\zeta+C\Pi\hat\theta}_{L^2(0,T;\Ys)} \right)^2  
\end{align*}
where $M=\min \left( \frac{\delta}{T^2\norm{C}_{\mathcal{L}(\Xs,\Ys)}^2\norm{B}_{L^2(0,T;\mathcal{L}(\Theta,\Xs))}^2},1 \right)$.
Finally, using the triangle inequality and the exact observability assumption to the last expression implies by \eqref{eq:ddk_fin},
\begin{align*}
\frac{d}{d\kappa} \frac12 \norm{\hat z^-(T)}_{\Xs}^2 \Big|_{\kappa=0} \le -M \norm{Ce^{A\cdot}\Delta\zeta}_{L^2(0,T;\Ys)}^2 \le -M\gamma^2 \norm{\Delta\zeta}_{\Xs}^2.
\end{align*}
Thus,
\[
\norm{g(\Delta\zeta)}_{\Xs}^2 \le \left( 1-2M\gamma^2\kappa+\mathcal{O}(\kappa^2) \right) \norm{\Delta\zeta}_{\Xs}^2
\]
from which the result follows by taking the square root and using the linear approximation $\sqrt{1+x}=1+x/2+\mathcal{O}(x^2)$.
\end{proof}

\begin{rmk}
The assumption $U_0 \ge \delta I$ is necessary since, in theory, there could exist $\Delta\zeta \in \Xs$ and $\hat\theta \in \Theta$ such that $C\Pi\hat\theta=-Ce^{A\cdot}\Delta\zeta$, meaning that the same output can be obtained with infinitely many initial state and parameter combinations. If it would then hold that $\bip{\hat\theta,U_0\hat\theta}=0$, the algorithm would be stuck. Obviously this is a rather pathological situation, so if the parameter is well identifiable, the method can be expected to work also with $U_0=0$. 
\end{rmk}

 \subsection{Optimization problem without feedback} \label{sec:nofb}
 
 Let us next consider the optimization problem
 \begin{equation} \label{eq:cost_nofb}
\min_{\zeta \in \Xs, \xi \in \Theta} \ip{\xi-\theta_0,U_0(\xi-\theta_0)}_{\Theta}+ \int_0^T\norm{y-C\hat z[\zeta,\xi]}_{\Ys}^2 d\tau
\end{equation}
where $\hat z[\zeta,\xi]$ is the solution to the open loop system
\[
\begin{cases}
\dot{\hat z}[\zeta,\xi]=A\hat z[\zeta,\xi]+B\xi, \\
\hat z[\zeta,\xi](0)=\zeta.
\end{cases}
\]
Obviously the method cannot work completely without the feedback term in the observer. However, as in \cite{OEBFN}, we can expect that if the gains $\kappa_j$ are taken to zero with a proper rate as the iterations advance, the estimates may converge to the minimizers of \eqref{eq:cost_nofb}. Since the mapping changes between iterations, we cannot utilize any fixed-point methods in proving the convergence, but a combination of the arguments in the proofs of Theorem~\ref{thm:attractive_lin} and \cite[Theorem~3.1]{OEBFN} will yield the result.

Denote again by $\zeta^o \in \Xs$ and $\theta^o \in \Theta$ the minimizer of \eqref{eq:cost_nofb} and denote by $z^o$ the corresponding optimal trajectory and $\chi:=y-Cz^o$. The optimality of $\zeta^o,\theta^o$ means that the residual $\chi$ satisfies
\begin{equation} \label{eq:residual}
\int_0^T e^{A^*s}C^*\chi(s)ds=0 \qquad \textup{and} \qquad U_0(\theta^o-\theta_0)+\int_0^T \Pi^*C^*\chi d\tau=0
\end{equation}
where $\Pi$ is the sensitivity operator satisfying $\dot\Pi=A\Pi+B$ and $\Pi(0)=0$. These conditions arise from differentiation of \eqref{eq:cost_nofb} with respect to $\zeta$ and $\xi$, respectively.

\begin{thm} \label{thm:nofb}
Assume that $A$ is skew-adjoint, the system is exactly observable, and $U_0 \ge \delta I>0$. Choose the observer gains in \eqref{eq:linear_joint} and \eqref{eq:bwd_lin} so that $\sum_{j=1}^{\infty} \kappa_j=\infty$ and $\sum_{j=1}^{\infty} \kappa_j^2<\infty$. Then the respective estimates $\hat z_j^-(T)$ and $\hat\theta_j$ for the initial state and the parameter given by \eqref{eq:linear_joint} and \eqref{eq:bwd_lin}, converge (in norm) to the optimal values $\zeta^o$ and $\theta^o$.
\end{thm}
\begin{proof}
Denote $\varepsilon^+:=z^o-\hat z_j^+$ (for some index $j$ that we omit in the proof) and $\varepsilon^-(t):=z^o(T-t)-\hat z_j^-(t)$. These error terms follow the dynamics
\[
\dot\varepsilon^+=(A-\kappa C^*C)\varepsilon^+-\kappa C^* \chi +B(\theta^o-\hat\theta)
\]
and
\[
\dot\varepsilon^-(t)=(-A-\kappa C^*C)\varepsilon^-(t)-\kappa C^* \chi(T-t)-B(T-t)(\theta^o-\hat\theta).
\]
As noted before, the old parameter estimate $\hat\theta_j$ has no effect on the next estimates, and since we are not interested in the trajectory $\hat z^+(t)$, we can put the same $\hat\theta$ --- which is the new parameter estimate --- in both equations. 

Following the proof of Theorem~\ref{thm:attractive_lin}, the backward part is re-organized to
\begin{align*}
\dot\varepsilon^-(t)=&-(A-\kappa C^*C)\varepsilon^-(t)+\kappa C^*\chi(T-t)-B(T-t)(\theta^o-\hat\theta) \\ & -2\kappa C^*C\varepsilon^-(t)-2\kappa C^*\chi(T-t).
\end{align*}
Neglecting the last two terms gives exactly the time-inverted equation for $\varepsilon^+$. In addition, if $\kappa=0$, the trajectories of the forward and backward equations are the same, namely $\varepsilon^-(T-t)=\varepsilon^+(t)=e^{At}\varepsilon^+(0)+\Pi(t)(\theta^o-\hat\theta)$. 
Further, it holds that
\[
\varepsilon^-(T)=\varepsilon^+(0)-2\kappa \int_0^T e^{-(A-\kappa C^*C)(T-s)}C^*\big( C\varepsilon^-(s)+\chi(T-s)\big)ds.
\]
The only difference here compared to \eqref{eq:confer} in the proof of Theorem~\ref{thm:attractive_lin} is the addition of the noise term $\chi$. However, this term appears in the derivative $\frac{d}{d\kappa}\varepsilon^-(T) \big|_{\kappa=0}$ and since $\kappa=0$, the contribution of this noise term vanishes because of \eqref{eq:residual} and $-A=A^*$.

Pursuing as in the proof of Theorem~\ref{thm:attractive_lin}  yields
\[
\norm{\varepsilon^-(T)}_{\Xs} \le \big(1-\alpha\kappa + \mathcal{O}(\kappa^2)\big)\norm{\varepsilon^+(0)}_{\Xs}
\]
where the $\mathcal{O}(\kappa^2)$-term is uniformly bounded with respect to $\varepsilon^+(0)$ and $\chi$. Using this bound repeatedly for $j=1,...,k$ gives
\[
\bnorm{\varepsilon_k^-(T)}_{\Xs} \le \prod_{j=1}^k (1-\alpha \kappa_j) \norm{\varepsilon_1^+(0)}_{\Xs}+\mathcal{O}(1) \sum_{j=1}^k\kappa_j^2 \prod_{i=j+1}^k (1-\alpha\kappa_i).
\]
The product terms are bounded by
\[
\prod_{i=j+1}^k (1-\alpha\kappa_i)=\exp \left( \sum_{i=j+1}^k \ln (1-\alpha \kappa_i) \right) \le \exp\left(-\alpha \! \sum_{i=j+1}^k \kappa_i \right)
\]
which converges to zero for any $j$ as $k \to \infty$, and so by the assumptions on the gains, $\bnorm{\varepsilon_k^-(T)}_{\Xs} \to 0$ as $k \to \infty$ concluding the proof.
\end{proof}

As in \cite{OEBFN}, this result can be made to hold also for essentially skew-adjoint and dissipative (ESAD) generators, meaning $\dom(A)=\dom(A^*)$ and $A+A^*=-Q$ for some bounded $Q \ge 0$. In that case the observer has to be corrected by replacing the feedback operator $\kappa C^*$ by $\kappa P(t)C^*$ in the forward observer and by $\kappa P(T-t)C^*$ in the backward observer where $P(t)=e^{At}e^{A^*t}$. For the required modifications in the proofs, see \cite[Lemma~3.1 and Theorem~3.2]{OEBFN}. The result holds for small enough $Q$. Typically we do not want to compute the full operator $P(t)$ as it would severely increase the computational cost of the method. In some cases the operator can be feasibly approximated. For an example of such approximation, see also Section~4 in \cite{OEBFN}, where it is shown that for the wave equation with constant dissipation $u_{tt}=\Delta u-\epsilon u_t$ with Dirichlet boundary conditions, it holds that $e^{A^*t}e^{At} \approx e^{-\epsilon t}I$. The modified proof relies on results on strongly continuous perturbations of semigroup generators, that can be found in \cite{Chen} by Chen and Weiss.

\section{Bilinear parameter estimation} \label{sec:bilin}

This section is devoted to studying the bilinear parameter estimation problem, namely consider a system whose dynamics are governed by equations
\begin{equation} \label{eq:bilin_syst}
\begin{cases}
\dot z=A(\theta)z+f+\eta, \\
y=Cz+\nu, \\
z(0)=z_0.
\end{cases}
\end{equation}
Here $f$ is a known load term and $\eta$ and $\nu$ are unknown model error and noise terms. We assume that the structure of the main operator is $A(\theta)=A_0+\Delta A \cdot \theta$ where the parameter belongs to a Hilbert space, $\theta \in \Theta$.

We assume that $A(\theta)$ is a skew-adjoint operator for all $\theta$. In many cases, the natural norm for the state space depends on the parameters. This is the case in our example on the inverse potential problem for the wave equation treated in Section~\ref{sec:wave}. In that case, it may be that $A(\theta)$ is skew-adjoint only if the state space is equipped with the norm computed using the same parameter $\theta$. Therefore we must be very careful when computing any norms, inner products, or adjoint operators in $\Xs$. In our proofs we only encounter small deviations from the optimal parameter, and we can use the $\Xs$-norm corresponding to this parameter. Therefore the convergence results also hold in this norm. Let us list the assumptions needed in this section. We remark, however, that the method can be used even though these assumptions are not satisfied as long as the BFN method is stable. The optimality results do not hold in that case, but the parameter estimate can nevertheless be expected to be reasonable. 

\begin{assu} \label{assu:bilin}
Make the following assumptions:
\begin{itemize}
\item[A1] The norm in $\Xs$ may depend on the parameter $\theta$ but norms corresponding to different parameters are equivalent.

\item[A2] The operator $A(\theta)$ is skew-adjoint when $\Xs$ is equipped with the norm corresponding to the parameter $\theta$.

\item[A3] The operator $\Delta A(\theta)=A(\xi+\theta)-A(\xi)$ is bounded in $\Xs$ and does not depend on $\xi\in\Theta$. It holds that $\norm{\Delta A(\theta)}_{\mathcal{L}(\Xs)} \le M \norm{\theta}_{\Theta}$ for some $M >0$. Denote the smallest possible $M$ by $\norm{\Delta A}$.

\item[A4] The norm in the output space $\Ys$ does not depend on $\theta$.
\end{itemize}
\end{assu}

The $j^{\textup{th}}$ (forward) iteration of the algorithm for the bilinear parameter estimation problem is given essentially by the same equations as in the linear case except for the computation of the sensitivity operator $\Pi$. That is, (set $\hat\theta_1=\theta_0$):
\begin{equation} \label{eq:bilin_joint}
 \begin{cases}
 \dot{\hat z}_j^+=A(\hat\theta_j)\hat z_j^+ +K_j(y-C\hat z_j^+)+f, \quad \hat z_j^+(0)=\hat z_j^-(T) \\
 \dot{\Pi}= (A(\hat\theta_j)-K_jC)\Pi + \Lambda \hat z_j^+, \quad \Pi(0)=0,     \\
 \dot{U}= \big(C\Pi\big)^*C\Pi, \quad U(0)=U_0, \\
 \dot{\xi}= \big(C\Pi\big)^*\big(y-C\hat z_j^+ \big), \quad \xi(0)=0, \\
 \hat\theta_{j+1}=\hat\theta_j-U(T)^{-1}\big(U_0(\hat\theta_j-\theta_0)-\xi(T) \big)   
  \end{cases}
 \end{equation}
where $\Lambda:\Xs \to \mathcal{L}(\Theta,\Xs)$ is defined for $h \in \Xs$ and $\xi \in \Theta$  through
\[
\big[ \Lambda h \big] \xi := \big[ \Delta A \cdot \xi \big] h.
\] 
   The backward state observer is defined "forward in time" by
   \begin{equation} \label{eq:bwd_bilin}
 \begin{cases}
 \dot{\hat z}_j^-(t)=-A(\hat\theta_{j+1})\hat z_j^-(t) +K_j(y(T-t)-C\hat z_j^-(t))-f(T-t), \\ \hat z_j^-(0)=\hat z_j^+(T)+\Pi_j(T)(\hat\theta_{j+1}-\hat\theta_j), \\ \hat z_{j+1}=\hat z^-(T).
 \end{cases}
 \end{equation} 
 Again with such definition, $\hat z_j^-(t)$ is an estimate for $z(T-t)$. Here the feedback term $K_j$ depends on the iteration but that does not need to be so. In the algorithm we have replaced $\Pi^*C^*$ by $\big(C\Pi\big)^*$ to avoid computing adjoints in the state space $\Xs$ whose inner product may depend on the parameter $\theta$.

Again we intend to show that under some assumptions, using the colocated feedback $K_j=\kappa C^*$ with constant gain $\kappa > 0$, the initial state and parameter estimates converge to the minimizer of the cost function
\begin{equation} \label{eq:cost_bilin}
\min_{\zeta \in \Xs, \xi \in \Theta} \ip{\xi-\theta_0,U_0(\xi-\theta_0)}_{\Theta}+ \int_0^T\norm{y-C\hat z[\zeta,\xi]}_{\Ys}^2 d\tau
\end{equation}
where $\hat z[\zeta,\xi]$ is the solution to
\[
\begin{cases}
\dot{\hat z}[\zeta,\xi]=A(\xi)\hat z[\zeta,\xi]+f+\kappa C^*(y-C\hat z[\zeta,\xi]), \\
\hat z[\zeta,\xi](0)=\zeta.
\end{cases}
\]
Obviously such nonlinear optimization problem may also have local minima, and as the Gauss--Newton method in general, the presented algorithm may get stuck into a local minimum. In general, it is difficult to say much about nonlinear least squares optimization problems. In \cite[Section~5.1.3]{Chavent} it is shown that for sufficiently smooth problems (such as the bilinear parameter estimation),  by bounding the set over which the minimization is performed, and adding a Tikhonov type regularization term with big enough coefficient, the problem will have a unique solution. However, in our case the minimization variable is $(\zeta,\xi)$ but the regularization term contains only the parameter $\xi$. Nevertheless, assuming exact observability, by a similar computation as in the proof of Lemma~\ref{lma:convex}, using similar techniques as in \cite[Section~5.1.3]{Chavent}, it can be shown that the regularization only in the parameter space suffices to guarantee the existence of a unique solution, if $U_0$ is big enough and the set of allowed variables is restricted to a small enough set.

In what follows, we simply assume the existence of a minimizer. In the computations below we only need the first order optimality conditions meaning that they hold also for local minima.
Denote again by $\zeta^o \in \Xs$ and $\theta^o \in \Theta$ the solution to the optimization problem \eqref{eq:cost_bilin}, the corresponding trajectory by $z^o=\hat z[\zeta^o,\theta^o]$, and the output residual by $\chi=y-Cz^o$.
As in the linear case, the method given by equations \eqref{eq:bilin_joint} and \eqref{eq:bwd_bilin} can be interpreted as a mapping $h:(\hat z_j,\hat\theta_j) \mapsto (\hat z_{j+1},\hat\theta_{j+1})$ and, again, the optimal point $(\zeta^o,\theta^o)$ is a fixed point of this mapping.
\begin{lma}
Make the assumptions A1--A4 in Assumption~\ref{assu:bilin}.
Then it holds that $h(\zeta^o,\theta^o)=(\zeta^o,\theta^o)$.
\end{lma}
\begin{proof}
If $\hat\theta_j=\theta^o$ and $\hat z_j=\zeta^o$ in \eqref{eq:bilin_joint}, it is clear that $\hat z^+=z^o$. Then using $y=Cz^o+\chi$, we obtain directly from \eqref{eq:bilin_joint}
\begin{align*}
\hat\theta_{j+1}&=\theta^o-U(T)^{-1}\left( U_0(\theta^o-\theta_0)-\int_0^T\big(C\Pi\big)^*\chi d\tau \right)=\theta^o
\end{align*}
because the term inside the parentheses is the gradient of the cost function in \eqref{eq:cost_bilin} with respect to the parameter $\xi$ at the optimum, and hence it is zero. That $\hat z^-(T)=\zeta^o$ is then shown exactly as in the proof of Lemma~\ref{lma:fix_lin}. Here it does not matter that the $\Xs$ inner product depends on the parameter since every trajectory is computed using $\theta^o$.
\end{proof}

In the main result of this section, it is shown that the optimum is also an attractive fixed point (under some restrictions on the residual $\chi$). The proof is based on a classical result (see, e.g., \cite[Chapter~22]{Ostrowski}) which says that if the spectral radius of the Fr\'echet derivative at a fixed point is smaller than one, the fixed point is attractive. 

The Fr\'echet derivative is computed with respect to the augmented variable containing the initial state and the actual parameter. Therefore the derivative operator has a 2-by-2 blockwise structure. We shall begin by showing an auxiliary perturbation result concerning the spectral radius of such block operator.
\begin{lma} \label{lma:rad}
Consider the block operator $\bm{D & E \\ F & G}$ from some product of two Hilbert spaces to itself where all the blocks are bounded operators in the respective spaces. Then the spectral radius of the block operator satisfies
\[
\rho \left( \sm{D & E \\ F & G} \right) \le \max\big(\rho(D),\rho(G)\big)+\sqrt{\norm{E}\norm{F}}.
\]
\end{lma}
\begin{proof}
Fix $\lambda \in \mathbb{C}$ so that $|\lambda| > \max\big(\rho(D),\rho(G)\big)$. Then $\sm{\lambda-D & -E \\ -F & \lambda-G}$ is invertible if $\lambda-D-E(\lambda-G)^{-1}F$ and $\lambda-G-F(\lambda-D)^{-1}E$ are invertible. We consider $E(\lambda-G)^{-1}F$ and $F(\lambda-D)^{-1}E$ as bounded perturbations to $D$ and $G$, respectively. 

If $B$ is bounded, it holds that $\rho(A+B) \le \rho(A)+\norm{B}$. Now
\[
\norm{E(\lambda-G)^{-1}F} \le \norm{E}\norm{F}\norm{(\lambda-G)^{-1}} \le \frac{\norm{E}\norm{F}}{|\lambda|-\rho(G)}
\]
and similarly $\norm{F(\lambda-D)^{-1}E} \le \frac{\norm{E}\norm{F}}{|\lambda|-\rho(D)}$. If
\begin{equation} \label{eq:sec}
\norm{E}\norm{F} < \big(|\lambda|-\rho(D)\big)\big(|\lambda|-\rho(G)\big),
\end{equation}
then $\norm{E(\lambda-G)^{-1}F} < |\lambda|-\rho(D)$ and $\norm{F(\lambda-D)^{-1}E} < |\lambda|-\rho(G)$, which implies invertibility of $\lambda-D-E(\lambda-G)^{-1}F$ and $\lambda-G-F(\lambda-D)^{-1}E$, and hence also $\sm{\lambda-D & -E \\ -F & \lambda-G}$. By solving  \eqref{eq:sec} with respect to $|\lambda|$, it follows that the inequality holds at least if
\[
|\lambda| > \frac{\rho(D)+\rho(G)}2 + \frac12 \sqrt{(\rho(D)-\rho(G))^2+4\norm{E}\norm{F}}.
\]
Using $\sqrt{a+b} \le \sqrt{a}+\sqrt{b}$ for $a,b \ge 0$, it is clear that this inequality follows from 
\[
|\lambda| > \max\big(\rho(D),\rho(G)\big)+\sqrt{\norm{E}\norm{F}}
\]
finishing the proof.
\end{proof}

Finally, we are ready for the main result of this section which is shown by finding an expression for the Fr\'echet derivative of the mapping $h$ by linearizing equations \eqref{eq:bilin_joint} and \eqref{eq:bwd_bilin} with respect to the optimum, and using then the previous lemma.

\begin{thm} \label{thm:attractive_bilin}
Make the assumptions A1--A4 in Assumption~\ref{assu:bilin} and, additionally, assume that the system is exactly observable for the parameter $\theta^o$. Assume that the residual $\chi$ is small (see Remark~\ref{rmk:chi}). Then the optimum $(\zeta^o,\theta^o)$ is an attractive fixed point of the algorithm \eqref{eq:bilin_joint} and \eqref{eq:bwd_bilin}.
\end{thm}

The proof is rather long and technical, and it does not contain any very interesting details. Therefore the proof is presented in Appendix~\ref{app:proof}. Here we only discuss the main ideas and difficulties of the proof. As mentioned, the idea in the proof is to use Lemma~\ref{lma:rad} on the Fr\'echet derivative of $h$ at the optimal point, which we denote by $\sm{D & E \\ F & G}$. The Fr\'echet derivative is formed by linearizing the algorithm equations \eqref{eq:bilin_joint} and \eqref{eq:bwd_bilin} with respect to the optimum by adding infinitesimal perturbations to the initial state and the parameter.

The spectral radius of $D$ is bounded from above by its operator norm. The norm, in turn, can be bounded as in the linear case in the proof of Theorem~\ref{thm:attractive_lin}. Unfortunately, this is not entirely straightforward even though the source estimation case can be seen as a linearization of the bilinear parameter estimation case in the vicinity of an observer trajectory. The reason for this additional complication is that even if the algorithm is run with the optimal initial state and parameter, when $\kappa>0$, the trajectories of the forward observer $\hat z^+$ and the backward observer $\hat z^-$ are not the same, resulting in an additional term $\Delta A(\Delta\theta)\big(\hat z^+(T-t)-\hat z^-(t)\big)$ in the error dynamics \eqref{eq:err_bilin}. In the linear case, the parameter perturbation is affecting the observer dynamics through the term $B\Delta\theta$, which does not depend on the observer trajectory. Therefore the corresponding term vanishes completely, because the contributions in the forward observer and backward observer cancel out each other (compare \eqref{eq:confer} and \eqref{eq:err_bilin}). The difference of the forward and backward observer trajectories depends on the residual term $\chi$, and therefore the operator $D$ has an additional term which is small if $\chi$ is small. Also $E$ is small if $\chi$ is small, and so the spectral radius of $\sm{D & E \\ F & G}$ is smaller than one for small enough residual $\chi$.

\begin{rmk} \label{rmk:chi}
An explicit bound for $\norm{\chi}_{L^2}$ guaranteeing the attractiveness of the fixed point is possible to obtain, but that would be quite restrictive. Therefore, the main contribution of this result is to show that there exists some neighborhood around the attainable set in the space $L^2(0,T;\Ys)$, such that if the given measurement $y$ is in this neighborhood, then the point in the attainable set that is closest to $y$, is an attractive fixed point for the method. However, the neighborhood is probably considerably bigger than what would be obtained by using some ``worst-case" bound. 

In addition, the algorithm can be regularized, for which the simplest way is to decrease the step size by modifying \eqref{eq:bilin_joint} with $\hat\theta_{j+1}=\hat\theta_j-\beta U(T)^{-1}\big(U_0(\hat\theta_j-\theta_0)-\xi(T) \big)$ for some coefficient $\beta \in (0,1]$.
\end{rmk}

\begin{rmk}
If the initial state estimate is bad, it may happen that the method is unstable. One way to improve stability is to start by running the BFN method one or more times without changing the parameter.

It is also possible to change the regularization operator $U_0$ between iterations. It may happen that the cost function \eqref{eq:cost_bilin} has local minima, and using a bigger $U_0$ in the first iterations may help the algorithm to find a minimum that is closer to the prior $\theta_0$.
\end{rmk}

\begin{rmk}
A natural question is that can the feedback then be taken to zero as in Section~\ref{sec:nofb} in the linear case. However, there are then no fixed-point theorems that could be used. In addition, if $\kappa$ is taken to zero in the previous theorem, then also the bound for $\norm{\chi}_{L^2}$ for which the result holds, approaches zero. Therefore it seems that a convergence result in the bilinear case corresponding to Theorem~\ref{thm:nofb} would be very hard to obtain.
\end{rmk}

\section{Examples}

We shall present two examples of which the first is for demonstrating the effect of the feedback in the output error minimization problem. If the model is erroneous, then keeping the observer gain positive may result in better parameter estimate. This example represents a low-dimensional oscillator, where the model error appears as slightly erroneous fundamental frequencies. The second example represents the inverse potential problem for the wave equation and its purpose is to illuminate the iterative state-parameter estimation method developed in this paper.

\subsection{The feedback effect in a simple source estimation problem} \label{sec:fb}

 Consider the source estimation problem for the following system representing an oscillating system, whose dynamics are governed by
 \[
 \begin{cases}
 \dot z= A_0 z+B\theta, \\
 z(0)=z_0, \\
 y=Cz
 \end{cases}
 \]
where $z  \in \mathbb{R}^5 \times \mathbb{R}^5$, and
\begin{align*}
&A_0 =\bm{ 0 & I \\ -\textup{diag}(1.05^2,1.94^2,2.95^2,4.02^2,5.03^2) & 0}, \\ &B=u(t) \bm{ 0 \\ I} \in \mathbb{R}^{10 \times 5},    \qquad C=\bm{0, \, 0,\, 0,\, 0,\, 0,\, 1,\, 1,\, 1,\, 1,\, 1},
\end{align*}
where $u(t)$ is an $\mathbb{R}$-valued signal that is a realization of a Gaussian process obtained as a solution of the auxiliary system
\begin{equation} \label{eq:input}
\begin{cases}
 dr = \sm{0 & 1 \\ -1.9 & -2.7}r \, dt+ \sm{0 \\ 1} dw(t), \\
 r(0) \sim N\left(0,\frac1{59.8} \sm{4 & 2 \\ 2 & 11} \right), \\ 
 u(t)=[4 \ 0]r(t)
\end{cases}
\end{equation}
where $w$ is a standard Brownian motion.
The initial condition for $r$ is chosen so that $r$, and hence also $u$, are stationary processes. This kind of input signal is regular enough (once continuously differentiable) but also a persistently exciting signal. The true parameter value is $\theta=[1,1,1,1,1]^T$.

The observer dynamics are given by
\begin{equation} \label{eq:source_obs}
\dot{\hat z}[\hat\theta]= A \hat z[\hat\theta]+B\hat\theta +\kappa C^*(y-C\hat z[\hat\theta]), \qquad \hat z[\hat\theta](0)=z_0
\end{equation}
where 
\[
A=\bm{ 0 & I \\ -\textup{diag}(1^2,2^2,3^2,4^2,5^2) & 0}
\]
so that the fundamental frequencies of the system are slightly off the true values. This corresponds roughly to a scenario where we model the system with one-dimensional wave equation with Dirichlet boundary conditions (five lowest eigenmodes), but in reality, the vibrating string is slightly inhomogeneous.

\begin{figure}[b]
\vspace{-2mm}
\mbox{\hspace{-15mm} 
\includegraphics[width=15cm]{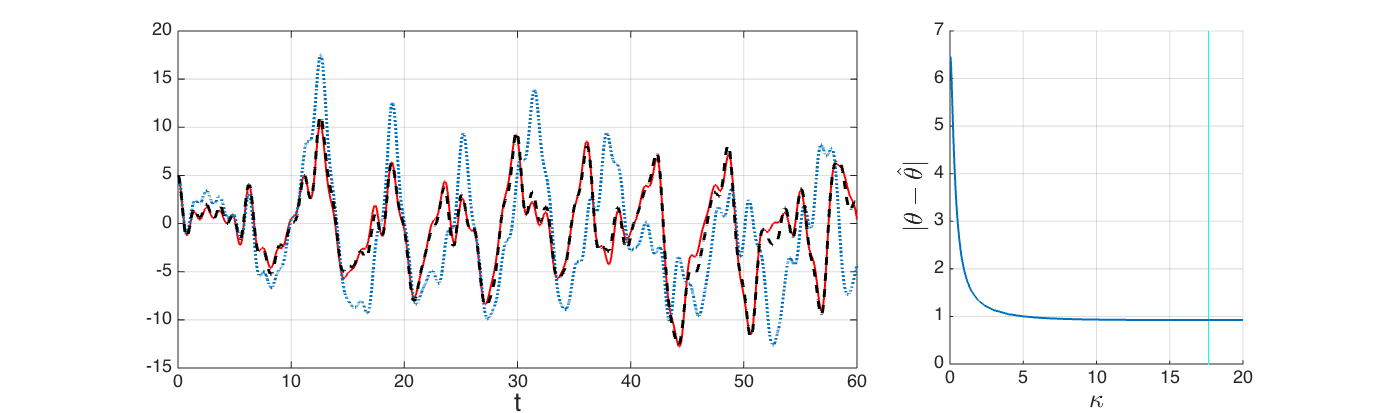}} 
\caption{Left: The true output in solid red line, the optimal observer output with $\kappa=0$ in dotted blue line, and the optimal observer output with $\kappa=1$ in dashed black line. Right: The Euclidian norm $\bnorm{\theta-\hat\theta}$ between the true parameter and the minimizer of  \eqref{eq:sourcemin} with different gain values~$\kappa$.}
\label{fig:outputs}\end{figure}

Consider then the minimization problem
\begin{equation} \label{eq:sourcemin}
\min_{\hat\theta \in \Theta} \bnorm{y-C\hat z[\hat\theta]}_{L^2(0,T;\Ys)}
\end{equation}
where $\hat z[\hat\theta]$ is the observer solution given by \eqref{eq:source_obs}. If the minimization is carried out using the open loop system, that is, with $\kappa=0$ in \eqref{eq:source_obs}, then even with the optimal parameter $\hat\theta$, the output $C\hat z[\hat\theta]$ from \eqref{eq:source_obs} is far from the measurement $y$, as seen on the left in Figure~\ref{fig:outputs} depicting one realization $u$ with $T=60$. Then also the optimal parameter is far from the true parameter. However, when the gain is increased, the output starts following the measurement much better, and also the optimal parameter value tends closer to the true parameter. The right panel of Figure~\ref{fig:outputs} shows the Euclidian distance between the minimizer $\hat\theta$ and the true parameter value $\theta$ for different gain values.
When $\kappa$ is increased from zero, the estimated parameter quickly improves a lot, but when $\kappa$ becomes greater than three, the minimizer does not change considerably. However, the gain value that minimizes the distance between the true and estimated parameters is $\kappa_{min} \approx 17.65$. When the gain grows beyond $\kappa_{min}$, the parameter estimate starts deteriorating slowly.


Altogether twenty realizations for $u$ were simulated. The exact shape of the $\bnorm{\theta-\hat\theta(\kappa)}$-curve seemed to depend somewhat on the realization, but the general behaviour was always similar.
The best parameter estimate was typically obtained with gain values above 10, but in one simulation we had $\kappa_{min} \approx 2.3$. In eight simulations, the norm between the true and estimated parameter seemed to reduce monotonically. Simulations were carried out up to $\kappa=200$ in these cases. However, it should be noted that the error never reduced significantly above $\kappa=5$. Interestingly, the gain value maximizing the  distance between the imaginary axis and the eigenvalues of $A-\kappa C^*C$ was as small as $\kappa \approx 0.56$ and the gain at which the first eigenmode became critically damped was $\kappa \approx 1.022$.



\begin{figure}[t]
\mbox{\hspace{-14mm} 
\includegraphics[width=15cm]{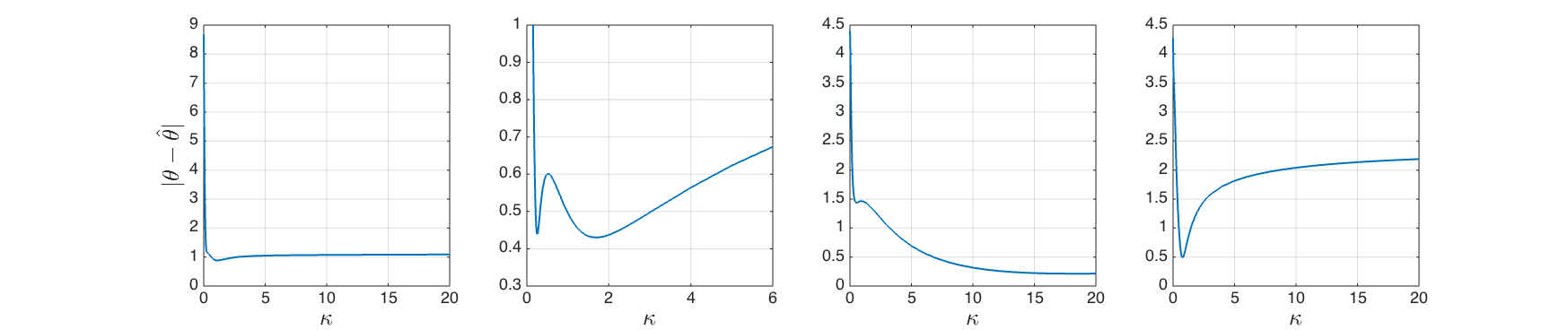}} 
\caption{Different types of behaviour for the parameter estimation error in the presence of output noise.}
\label{fig:noisy}\end{figure}

Some simulations were carried out with noisy output. The output noise process is a realization of the Ohrnstein--Uhlenbeck process obtained as the solution of the stochastic differential equation
\begin{equation} \label{eq:OU}
dv(t)=-\frac{25}{2}v(t) \, dt+\frac{5}{2}dw(t), \qquad v(0) \sim N(0,1/4),
\end{equation}
where $w$ is a standard Brownian motion. The initial condition is chosen so that the noise process is stationary. With noisy output, there was more variation in the behaviour of the parameter estimate.
Figure~\ref{fig:noisy} shows the difference between the true and estimated parameter for different gain values in four different scenarios. The leftmost picture shows a very nice case, where the parameter quickly improves a lot when the observer gain is increased from zero. In addition, the error does not really even grow when the gain is further increased. The second picture shows a case with multiple local minima and an error that starts increasing rather quickly when the gain is increased more. The third picture shows again a fairly nice case, but there the best parameter estimate is obtained with quite high (>10) gain values. In contrast, the fourth picture shows a case where the estimate is good on a narrow gain region (the error is below one between $\kappa \in (0.47,1.44)$). Let us remark that there are no exact rules for finding a suitable observer gain to be used in parameter estimation. However, this is a problem that always emerges in connection of observers. In ten simulations with different input and noise realizations, the gain value for which the state observer performed optimally (minimizing $\norm{z-\hat z}_{L^2(0,T;\Xs)}$) was between 0.8 and 1.2.

\subsection{Wave equation inverse potential problem} \label{sec:wave}
 
In the second example, we do not concentrate on the effect of the feedback, but more on the performance of the presented algorithm applied to a classical inverse potential problem for the wave equation. This problem is widely studied in the PDE analysis community, see, e.g.,  \cite{Baudouin} and \cite{Imanuvilov}. The dynamics of the considered system are governed by the one-dimensional wave equation with potential,
\begin{align*}
&\frac{\partial^2}{\partial t^2}u(t,x)=\frac{\partial^2}{\partial x^2}u(t,x)-\theta(x)u(t,x)+\tilde f(t,x), \\
& u(t,0)=u(t,1)=0
\end{align*}
where $\tilde f \in L^2\big(0,T;L^2(0,1)\big)$ is a known load term and the potential term $\theta$ is the parameter to be identified. The system is written as a first order system through introduction of the augmented state vector $z=[u \ u_t]^T \in \Xs$, where the state space is $\Xs=H_0^1[0,1] \times L^2(0,1)$ equipped with norm
\[
\norm{\bm{z_1 \\ z_2}}_{\Xs}^2:=\int_0^1 \bigg( \!\! \left(\frac{\partial z_1}{\partial x} \right)^{\! 2}+\theta z_1^2+z_2^2 \bigg)dx.
\]
For the augmented state vector $z$ we have the standard dynamics equations
\[
\dot z(t)=A(\theta)z(t)+f(t)
\]
where $f(t)=[0 \ \tilde f(t)]^T \in L^2(0,T;\Xs)$ and $A(\theta)=\bm{ 0 & I \\ \Delta - \theta & 0}$ where $\theta$ is regarded as a multiplication operator. The measurement obtained from the system consists of a partial velocity field measurement $u_t(t,x)$ for $x \in [0,0.1]$ and the averages
\[
\int_{0.05+0.05j}^{0.1+0.05j} u_t(t,x)dx
\] 
for $j=1,...,18$. Thus the output space is $\Ys=L^2(0,0.1) \times \mathbb{R}^{18}$ and the output operator $C:\Xs \to \Ys$ is bounded.

The load function is given by $\tilde f(t,x)=f_1(t)b_1(x)+f_2(t)b_2(x)+f_3(t)b_3(x)$ where $b_1(x)=(1-x)\sin(\pi x)$, $b_2(x)=7x^2(1-x)$, and $b_3(x)=\frac{\sin^2(6\pi x)}{x}$, and the time components $f_j$ for $j=1,2,3$ are three different realizations of the Gaussian process that was used also in the first example, described in \eqref{eq:input}. The output is given by
\[
y(t)=Cz(t)+\bm{\chi_1 \\ \chi_2}
\]
where $\chi_1$ is a noise process taking values in $L^2(0,0.1)$ and it is given by
\begin{equation} \label{eq:pot_noise}
\chi_1(t,x)=\sum_{j=1}^{\infty}\frac1{j}\sin\left( \frac{(2j-1)\pi}{0.2}x \right)v_j(t)
\end{equation}
where $v_j$'s are independent realizations of the Ohrnstein--Uhlenbeck process \eqref{eq:OU} used also in the previous example, multiplied by 0.1. Also $\chi_2$ consists of 18 independent realizations of the same process multiplied by 0.1.

Let us move on to the iterative joint state and parameter estimation algorithm. The operator $C^*C$ in the observer is given by
\[
C^*C \bm{g \\ h} = \bm{0 \\ h \mathbbm{1}_{[0,0.1]}}+\sum_{k=1}^{18}\bm{0 \\ \mathbbm{1}_{J_k}}\int_{J_k} h(t,x)dx
\]
where $\mathbbm{1}$ is the characteristic function of an indicated set and $J_k=[0.05+0.05k,0.1+0.05k]$. The time-dependent sensitivity operator $\Pi(t):\Theta \to \Xs$ is defined for $\xi \in \Theta$ as the solution at time $t$ of the equation
\[
\frac{d}{dt}\big(\Pi\xi \big)(t)=(A(\theta)-\kappa C^*C)\big(\Pi\xi\big)(t)-\bm{0 & 0 \\ \xi & 0}\hat z(t), \qquad \big(\Pi\xi\big)(0)=0
\]
where $\theta$ is the parameter at which the sensitivity is computed (that is used to compute the estimate $\hat z$) and $\xi$ is understood as a multiplication operator.

\begin{table}[t]
  \centering
\vspace{-2mm}
\caption{Estimation errors of the parameter ($L^1(0,1)$-norm), initial displacement ($L^2(0,1)$-norm of the $x$-derivative), and initial velocity ($L^2(0,1)$-norm).}
    \vspace{2mm}
  \begin{tabular}{clll}
    \hline
    \hline
 Iteration & Param. & Displ. & Vel.  \\
 & $\cdot 10^{-1}$ & $\cdot 10^{-3}$ & $\cdot 10^{-3}$ \\ 
 \hline
1 & 2.988 & 15.511 & 9.890 \\
2 & 1.588 & 6.463 & 5.344 \\
3 & 1.551 & 6.602 & 5.726 \\
4 & 1.550 & 6.587 & 5.780
 \\
   \hline
    \hline
  \end{tabular}
  \label{tab:convergence}
\end{table}

The algorithm was tested with a potential term $\theta$ that is zero at $x < 0.4$ or $x> 0.85$ and $\theta(x)=2$ for $x \in [0.45,0.8]$ with linear slopes at $[0.4,0.45]$ and $[0.8,0.85]$ (see Figure~\ref{fig:potential}). The initial state of the system is 
\[
\begin{cases}
u(0,x)=0.5x^{0.8}\sin(\pi x)+\sin(4\pi x), \\
u_t(0,x)=-8x(1-x)+1.6\sin(2\pi x),
\end{cases}
\]
and the initial guesses for both the parameter and the initial state are zero. Also the parameter prior is zero, that is, $\theta_0=0$ in the algorithm equations \eqref{eq:bilin_joint}. 
In this simulation the observer gain is $\kappa=2$, and the Tikhonov regularization operator is $U_0=-6\cdot 10^{-5}\Delta+1.5\cdot 10^{-5}I$ so that
\[
\ip{\xi,U_0\xi}_{\Theta}=6\cdot 10^{-5} \norm{d_x\xi}_{L^2(0,1)}^2+1.5\cdot 10^{-5}\norm{\xi}_{L^2(0,1)}^2.
\]
This is a sort of a \emph{smoothness prior}, meaning that the regularization term penalizes big jumps and oscillations in the solution.

Obviously, a discretization of both the state and the parameter space is required in a practical implementation. In this example, we use the same discretization for both, namely an equispaced finite element mesh with discretization interval 0.01 and piecewise linear basis functions (so called hat functions). In the discretized version, the output part $L^2(0,0.1)$ consists of ten first components of the discretized velocity vector. Therefore in the discretized version of the output noise $\chi_1$ given in \eqref{eq:pot_noise} we take into account ten first terms of the sum.

The algorithm converges quickly, as can be seen from Table~\ref{tab:convergence}. In the first iteration, the erroneus initial state causes an error in the parameter estimate. However, the initial state estimate converges more or less in the first iteration (recall that the first parameter estimate is used already in the first backward phase). After that, also the parameter estimate converges fast, and so no significant improvement happens after two iterations. The initial state and the parameter estimates after two iterations are shown in Figure~\ref{fig:potential}. Out of curiosity, we also took $\kappa$  slowly to zero between iterations, even though our theoretical results do not cover this strategy in the bilinear estimation problem. The state and parameter estimates seemed to converge, and at the limit, the errors were (cf. Table~\ref{tab:convergence}) 1.464$\cdot 10^{-1}$ for the parameter, 3.118$\cdot 10^{-3}$ for the initial displacement, and 3.539$\cdot 10^{-3}$ for the initial velocity. In particular, the initial state estimate has improved compared to the iterations with $\kappa=2$. However, one should keep in mind that in this example, the observer dynamics did not contain any modeling errors.

\begin{figure}[t]
\mbox{\hspace{-14mm}
\includegraphics[width=15cm]{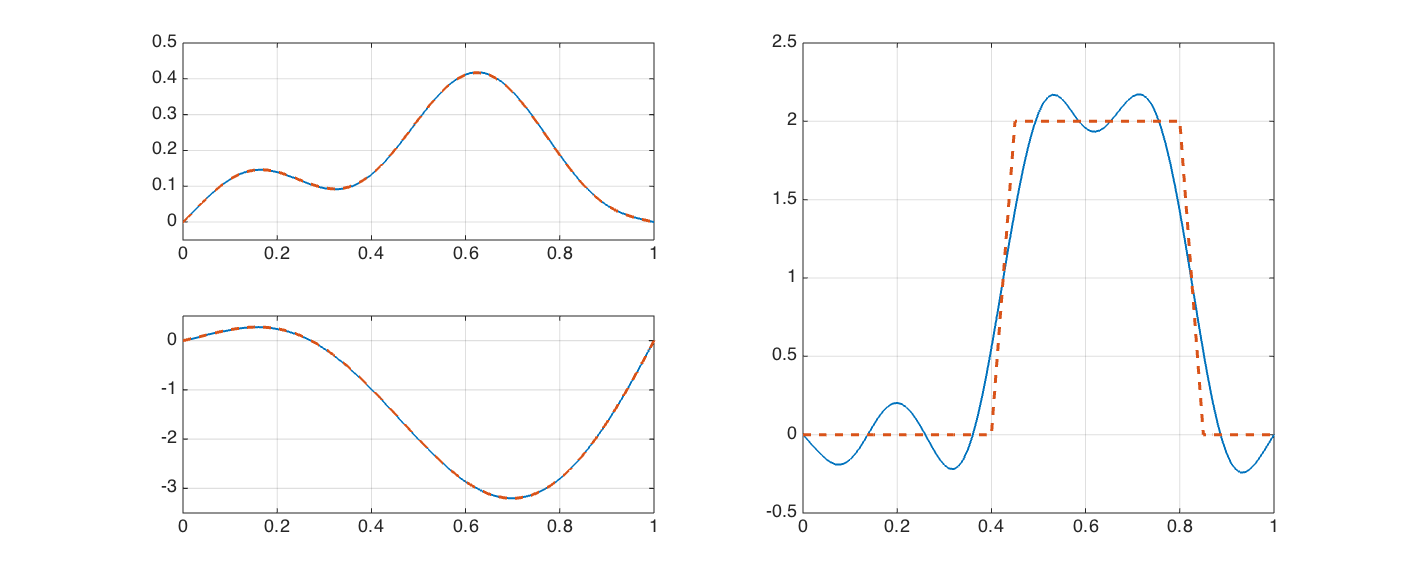}} 
\caption{The initial displacement estimate (top left) and the initial velocity estimate (bottom left) together with the potential estimate (right) after two iterations. The estimates are shown with the solid blue lines and the true ones are the dashed red lines. The initial state estimate and the true initial state can barely be distinguished.}
\label{fig:potential}\end{figure}







\appendix

\section{} \label{app:proof}

\begin{proof}[Proof of Theorem \ref{thm:attractive_bilin}]

The result follows by applying Lemma \ref{lma:rad} to the Fr\'echet derivative of $h$ at the optimal point, denoted by $\sm{D & E \\ F & G}$ in accordance with the notation used in Lemma \ref{lma:rad}. Let us compute and bound each block one at a time, by linearizing the algorithm equations \eqref{eq:bilin_joint} and \eqref{eq:bwd_bilin}. The linearization is carried out by considering infinitesimal perturbations in turn to the optimal parameter, and optimal initial state. That is, we consider $h(\zeta^o+\Delta\zeta,\theta^o)$ and $h(\zeta^o,\theta^o+\Delta\theta)$ and we shall only track the first-order components with respect to $\Delta\zeta$ and $\Delta\theta$.

The fact that $A(\theta)$ is skew-adjoint will be used several times, but only for the optimal parameter $\theta=\theta^o$. Therefore the inner product and norm in $\Xs$ are the ones corresponding to $\theta^o$.

{\bf Block $D$:}
Set $\zeta=\zeta^o+\Delta \zeta$ where $\Delta\zeta \in \Xs$ is an infinitesimal vector, and $\xi=\theta^o$. The derivative $D$ can be obtained by a modification of the proof of Theorem~\ref{thm:attractive_lin}.
Consider the error terms $\varepsilon^+:=\hat z_j^+-z^o+\Pi(\hat\theta-\theta^o)$ where $\hat\theta$ is given by the mapping $h(\zeta^o+\Delta\zeta,\theta^o)$, and $\varepsilon^-(t):=\hat z_j^-(t)-z^o(T-t)$ (for some index $j$ that will be omitted in the proof). The dynamics of these error terms are given by
\[
\begin{cases}
\dot{\varepsilon}^+=(A(\theta^o)-\kappa C^*C)\varepsilon^++\Delta A(\hat \theta-\theta^o)\hat z^+, \\ \varepsilon^+(0)=\Delta\zeta,
\end{cases}
\]
and
\[
\begin{cases}
\dot\varepsilon^-(t)=(-A(\theta^o)-\kappa C^*C)\varepsilon^-(t)-\Delta A(\hat\theta-\theta^o)\hat z^-(t)+2\kappa C^*\chi(T-t), \\
\varepsilon^-(0)=\varepsilon^+(T).
\end{cases}
\]
Following the proof of Theorem~\ref{thm:attractive_lin}, the backward equation is written in the form\begin{align} \label{eq:err_bilin}
\dot\varepsilon^-(t)=&(-A(\theta^o)+\kappa C^*C)\varepsilon^-(t)-\Delta A(\hat\theta-\theta^o)\hat z^+(T-t)  \\ \nonumber & -2\kappa C^*C \varepsilon^-(t)+\Delta A(\hat\theta-\theta^o)\big(\hat z^+(T-t)-\hat z^-(t)\big)+2\kappa C^*\chi(T-t),
\end{align}
where the first two terms comprise the time-inverted forward equation. Again by the semigroup perturbation formula, it holds that
\begin{align} \label{eq:eT}
\varepsilon^-(T) & =\Delta\zeta -2\kappa \int_0^T e^{-(A(\theta^o)-\kappa C^*C)(T-s)}C^* \big(C\varepsilon^-(s)-\chi(T-s)\big) ds \! \\ \nonumber
& \ \ + \! \int_0^T \! e^{-(A(\theta^o)-\kappa C^*C)(T-s)}\Delta A(\hat\theta-\theta^o)\big(\hat z^+(T-s)-\hat z^-(s)\big)ds. \!
\end{align} 
The contribution of $\chi$ vanishes from the first integral because of the optimality condition.
Compared with \eqref{eq:confer} in the proof of Theorem~\ref{thm:attractive_lin}, here is one additional term arising from the fact that the ``load term" in the error dynamics equations depends on the trajectory $\hat z^{\pm}$. 
It will be seen in a moment that $\hat\theta-\theta^o=\mathcal{O}(\norm{\Delta\zeta})$. Therefore, in the additional second integral in \eqref{eq:eT} it suffices to take into account only the zeroth order part (w.r.t. $\Delta\zeta$) of the difference $\hat z^+(T-t)-\hat z^-(t)$, meaning that we can consider the trajectories corresponding to the optimal point $(\zeta^o,\theta^o)$. It holds that
\[
\frac{d}{dt}\left( \hat z^+(T-t)-\hat z^-(t) \right)=\left( -A(\theta^o)-\kappa C^*C \right)\left( \hat z^+(T-t)-\hat z^-(t) \right)-2\kappa C^*\chi(T-t).
\]
The condition at $t=0$ is $\hat z^-(0)=\hat z^+(T)$, since $h(\zeta^o,\theta^o)=(\zeta^o,\theta^o)$ and so there is no contribution from the term $\Pi_j(T)(\hat\theta_{j+1}-\hat\theta_j)$ in \eqref{eq:bwd_bilin}. From the above equation, it is clear that when $\kappa=0$,  this additional term is zero, that is, $\hat z^+(T-t)=\hat z^-(t)$. For the proof, we need the derivative $\frac{d}{d\kappa}\varepsilon^-(T)\big|_{\kappa=0}$. The derivative of the first integral term in \eqref{eq:eT} is as in Theorem~\ref{thm:attractive_lin}, but for the derivative of the second integral, we need the derivative of the trajectory discrepancy. So denote $\hat\varepsilon(t):=\hat z^+(T-t)-\hat z^-(t)$. Using the semigroup perturbation formula yields
\[
\hat\varepsilon(t)=-\kappa \int_0^t e^{-A(\theta^o)(t-s)} C^* \big( C \hat\varepsilon(s)+2\chi(T-s)\big)ds
\]
from which it directly follows (recalling that $\hat\varepsilon=0$ if $\kappa=0$)
\begin{equation} \label{eq:traj_disc}
\frac{d}{d\kappa} \hat\varepsilon(t) \Big|_{\kappa=0}=-2\int_0^t e^{-A(\theta^o)(t-s)}C^*\chi(T-s)ds.
\end{equation}

Using \eqref{eq:eT} and $-A(\theta^o)=A(\theta^o)^*$ in $\Xs[\theta^o]$ as in the proof of Theorem~\ref{thm:attractive_lin}, we conclude
\begin{align*}
\frac{d}{d\kappa} \frac12 \norm{\varepsilon^-(T)}_{\Xs[\theta^o]}^2 \Big|_{\kappa=0}& = -2 \int_0^T \ip{ Ce^{A(\theta^o)(T-s)}\Delta\zeta,C\varepsilon^-(s) }_{\! \Ys}ds+e
\end{align*}
where $e$ is the additional error term caused by the $\kappa$-derivative of the last term in \eqref{eq:eT}. The effect of the first term can now be computed exactly as in the proof of Theorem~\ref{thm:attractive_lin}, but because there everything was computed without the load terms, one should note that there we had $z^o=0$ and $\theta^o=0$. Thus one should replace $\hat z^{\pm}$ by $\varepsilon^{\pm}$ and $\hat\theta$ by $\hat\theta-\theta^o$ in the proof.

The additional error term $e$ is given by
\begin{align*}
e=&\ip{\Delta\zeta,\int_0^T e^{-A(\theta^o)(T-s)}\Delta A(\hat\theta-\theta^o)\frac{d\hat\varepsilon}{d\kappa}(s)ds} \\
=& \int_0^T \ip{e^{A(\theta^o)(T-s)}\Delta\zeta, \Delta A(\hat\theta-\theta^o)\frac{d\hat\varepsilon}{d\kappa}(s)}ds.
\end{align*}
By Cauchy--Schwartz inequality (see also A3 in Assumption~\ref{assu:bilin}), it holds that
\begin{align*}
|e| &\le \norm{e^{A(\theta^o)(T-\cdot)}\Delta\zeta}_{L^2(0,T;\Xs)}\norm{\Delta A(\hat\theta-\theta^o)\frac{d\hat\varepsilon}{d\kappa}}_{L^2(0,T;\Xs)} \\
& \le \sqrt{T} \norm{\Delta\zeta}_{\Xs[\theta^o]} \norm{\Delta A}\bnorm{\hat\theta-\theta^o}\norm{\frac{d\hat\varepsilon}{d\kappa}}_{L^2(0,T;\Xs)}.
\end{align*}
Applying Young's inequality for convolutions to \eqref{eq:traj_disc} gives $\norm{\frac{d\hat\varepsilon}{d\kappa}}_{L^2(0,T;\Xs)} \le 2T\norm{C}\norm{\chi}_{L^2(0,T;\Ys)}$.
Since $\bm{C\Pi \\ I}(\hat\theta-\theta^o)$ is the orthogonal projection of $\bm{Ce^{A(\theta^o)\cdot}\Delta\zeta \\ 0}$ onto the subspace $\bm{C\Pi \\ I}\Theta \subset L^2(0,T;\Ys)\times \Theta$, it holds that
\begin{equation} \label{eq:parest}
\bnorm{\hat\theta-\theta^o}_{\Theta} \le \sqrt{\frac{T}{2\delta}}\norm{C} \norm{\Delta\zeta}_{\Xs}.
\end{equation}
Finally we can deduce
\[
|e| \le \sqrt{2/\delta}T^2\norm{C}^2\norm{\Delta A}\norm{\chi}_{L^2(0,T;\Ys)} \norm{\Delta\zeta}^2=:M\norm{\chi}_{L^2(0,T;\Ys)} \norm{\Delta\zeta}^2, 
\]
and further
\[
\norm{\varepsilon^-(T)} \le \left( 1-\left(\alpha-\frac12 M\norm{\chi}_{L^2(0,T;\Ys)}\right)\kappa+\mathcal{O}(\kappa^2) \right) \norm{\Delta\zeta}.
\]
This inequality gives a bound for the norm of the first block $D$ in the Fr\'echet derivative, and for $\norm{\chi}_{L^2(0,T;\Ys)}$ sufficiently small, the norm --- and hence also the spectral radius --- is smaller than one. 

{\bf Block $F$:}
From \eqref{eq:parest}, it follows that $\norm{F} \le \sqrt{\frac{T}{2\delta}}\norm{C}$.

{\bf Block $G$:}
Set $\zeta=\zeta^o$ and $\xi=\theta^o+\Delta\theta$ where $\Delta\theta \in \Theta$ is an infinitesimal vector. To keep notation consistent with \eqref{eq:bilin_joint} and \eqref{eq:bwd_bilin}, denote $(\hat z^-(T),\hat\theta)=h(\zeta,\xi)$. Denoting $\varepsilon^+:=z^o-\hat z^+ -\Pi(\theta^o-\xi)=z^o-\hat z^++\Pi\Delta\theta$, it is obtained directly from \eqref{eq:bilin_joint}, that 
\begin{align*}
\hat\theta&=\xi-U(T)^{-1}\left( U_0(\xi-\theta_0)-\int_0^T\big(C\Pi\big)^*(y-C\hat z^+)d\tau \right)     \\ & = \theta^o+\Delta\theta-U(T)^{-1}\left( U_0(\theta^o+\Delta\theta-\theta_0)- \! \int_0^T \!\! \big(C\Pi\big)^*(C\varepsilon^+-C\Pi\Delta\theta+\chi)d\tau \right) \\ &= \theta^o-U(T)^{-1}\left( U_0(\theta^o-\theta_0)- \int_0^T  \big(C\Pi\big)^*(C\varepsilon^++\chi) d\tau \right) \\  
&=\theta^o+U(T)^{-1} \int_0^T \big( \big(C\Pi\big)^*C\varepsilon^+ +\big(\big(C\Pi\big)^*-\big(C\Pi^o\big)^*\big)\chi \big) d\tau
\end{align*}
where $\Pi^o$ is the sensitivity operator corresponding to the optimal parameter and optimal trajectory, and it is given by $\dot \Pi^o=(A(\theta^o)-\kappa C^*C)\Pi^o+\Lambda z^o$, $\Pi^o(0)=0$. The last equality in the above computation follows from the property
 $U_0(\theta^o-\theta_0)-\int_0^T \big(C\Pi^{o*}\big)\chi d\tau=0$ which holds because it is the gradient  of the cost function with respect to $\xi$ at the optimum. For the error term $\varepsilon^+$ it holds that
\[
\dot\varepsilon^+=(A(\theta^o)-\kappa C^*C)\varepsilon^++\Delta A(\Delta\theta)\Pi \Delta\theta, \qquad \varepsilon^+(0)=0
\]
and therefore $\varepsilon^+ =\mathcal{O}(\norm{\Delta\theta}^2)$ and so it can be neglected as we are only interested in the linear terms with respect to small variations $\Delta\theta$. The second term in the integrand in the above expression for $\hat\theta$ is the usual first order Gauss--Newton step error that depends on the optimal residual $\chi$, and the curvature of the attainable set through $\Pi-\Pi^o$. For this term we have
\begin{align*}
\frac{d}{dt}(\Pi-\Pi^o) &=(A(\theta)-\kappa C^*C)\Pi-(A(\theta^o)-\kappa C^*C)\Pi^o + \Lambda (\hat z^+-z^o) \\
&=(A(\theta^o)-\kappa C^*C)(\Pi-\Pi^o) +\Delta A(\Delta\theta)\Pi + \Lambda (\hat z^+-z^o)
\end{align*}
which has a first order dependence on $\Delta\theta$, and so we deduce that $\bnorm{\hat\theta-\theta^o} \le M \norm{\chi}_{L^2(0,T;\Ys)}\norm{\Delta\theta}$ for some $M>0$, implying $\norm{G} \le  M \norm{\chi}_{L^2(0,T;\Ys)}$.

{\bf Block $E$:}
Finally we consider the effect of the infinitesimal perturbation in the parameter on the initial state component of $h$. Define $e^+:=z^o-\hat z^+-\Pi(\hat\theta-\xi)$ and $e^-(t):=z^o(T-t)-\hat z^-(t)$. The forward error satisfies
\[
\dot e^+=(A(\theta^o)-\kappa C^*C)e^++\Delta A(\theta^o-\hat\theta)\hat z^+-\Delta A(\Delta\theta)\Pi(\hat\theta-\xi)
\]
with $e^+(0)=0$. Now $\hat\theta-\xi=\hat\theta-\theta^o-\Delta\theta$ and so the latter term is of second order and so $e^+(T)=\mathcal{O}\big( \bnorm{\hat\theta-\theta^o}\big)+\mathcal{O}\big( \norm{\Delta\theta}^2 \big)$. The backward error satisfies
\[
\dot e^-(t)=(-A(\theta^o)-\kappa C^*C)e^-(t)-2\kappa C^*\chi(T-t)+\Delta A(\hat\theta-\theta^o)
\]
with $e^-(0)=e^+(T)$. The contribution of the residual term at $t=T$ is
\[
\int_0^T e^{(-A(\theta^o)-\kappa C^*C)(T-s)}C^*\chi(T-s)ds=\int_0^T e^{(A(\theta^o)^*-\kappa C^*C)s}C^*\chi(s)ds=0
\]
because of the optimality condition. Therefore it holds that $e^-(T)=\mathcal{O}\big( \bnorm{\hat\theta-\theta^o}\big)+\mathcal{O}\big( \norm{\Delta\theta}^2 \big)$ and so also $\norm{E}=\mathcal{O}(\norm{\chi}_{L^2})$.

{\bf Conclusion:} By Lemma~\ref{lma:rad}, it holds that
\[
\rho\left( \sm{D & E \\ F & G}\right) \le \left( 1-\left(\alpha-M_1\norm{\chi}_{L^2(0,T;\Ys)}\right)\kappa+\mathcal{O}(\kappa^2) \right) +M_2\sqrt{\norm{\chi}_{L^2(0,T;\Ys)}}
\] 
where $\alpha$ is given in Theorem~\ref{thm:attractive_lin} and $M_1$ and $M_2$ are positive constants that are collected from the proof. For small gains $\kappa$, the spectral radius is smaller than one if $\norm{\chi}_{L^2(0,T;\Ys)}$ is small enough. 
\end{proof}

\bibliographystyle{plain}
\bibliography{bibli}

\begin{thebibliography}{10}

\bibitem{OEBFN}
A.~Aalto.
\newblock Output error minimizing back and forth nudging method for initial
  state recovery.
\newblock {\em Systems and Control Letters}, 94:111--117, 2016.

\bibitem{collocated_book}
K.~Ammari and S.~Nicaise.
\newblock {\em Stabilization of Elastic Systems by Collocated Feedback}.
\newblock Lecture Notes in Mathematics 2124, Springer, 2015.

\bibitem{AB05}
D.~Auroux and J.~Blum.
\newblock Back and forth nudging algorithm for data assimilation problems.
\newblock {\em Comptes Rendus de l'Academie des Sciences, S\'erie I
  (Mathematique)}, 340:873--878, 2005.

\bibitem{AB08}
D.~Auroux and J.~Blum.
\newblock A nudging-based data assimilation method: the back and forth nudging
  ({BFN}) algorithm.
\newblock {\em Nonlinear Processes in Geophysics}, 15:305--319, 2008.

\bibitem{Baudouin}
L.~Baudouin, M.~{de Buhan}, and S.~Ervedoza.
\newblock Global {C}arleman estimates for waves and applications.
\newblock {\em Communications in Partial Differential Equations},
  38(5):823--859, 2013.

\bibitem{Baumeister}
J.~Baumeister, W.~Scondo, {M.A.} Demetriou, and {I.G.} Rosen.
\newblock On-line parameter estimation for infinite-dimensional dynamical
  systems.
\newblock {\em SIAM Journal of Control and Optimization}, 35(2):678--713, 1997.

\bibitem{Chapelle_wave}
D.~Chapelle, N.~C{\^i}ndea, M.~{De Buhan}, and P.~Moireau.
\newblock Exponential convergence of an observer based on partial field
  measurements for the wave equation.
\newblock {\em Mathematical Problems in Engineering}, 2012, Article ID: 581053,
  2012.

\bibitem{MoireauHinf}
D.~Chapelle, P.~Moireau, and P.~Le Tallec.
\newblock Robust filtering for joint state-parameter estimation in distributed
  mechanical systems.
\newblock {\em Discrete and Continuous Dynamical Systems, Series A},
  23(1--2):65--84, 2009.

\bibitem{ParticleSHM}
E.~Chatzi and A.~Smyth.
\newblock Particle filter scheme with mutation for the estimation of
  time-invariant parameters in structural health monitoring applications.
\newblock {\em Structural Control and Health Monitoring}, 20(7):1081--1095,
  2013.

\bibitem{Chavent}
G.~Chavent.
\newblock {\em Nonlinear Least Squares for Inverse Problems}.
\newblock Springer, 2009.

\bibitem{Chen}
J.-H. Chen and G.~Weiss.
\newblock Time-varying additive perturbations of well-posed linear systems.
\newblock {\em Mathematics of Control, Signals, and Systems}, 27:149--185,
  2015.

\bibitem{Cox}
S.~Cox and E.~Zuazua.
\newblock The rate at which energy decays in a damped string.
\newblock {\em Communications in Partial Differential Equations},
  19(1-2):213--243, 1994.

\bibitem{Curtain_Weiss}
R.~Curtain and G.~Weiss.
\newblock Exponential stabilization of well-posed systems by colocated
  feedback.
\newblock {\em SIAM Journal of Control and Optimization}, 45(1):273--297, 2006.

\bibitem{CZ}
R.~Curtain and H.~Zwart.
\newblock {\em An Introduction to Infinite Dimensional Linear Systems Theory}.
\newblock Texts in Applied Mathematics 21, Springer--Verlag, New York, 1995.

\bibitem{Schur}
D.~Cvetkovi\'c-Ili\'c, D.~Djordjevi\'c, and V.~Rako\v{c}evi\'c.
\newblock Schur complements in {C}$^*$-algebras.
\newblock {\em Mathematische Nachrichten}, 278(7--8):808--814, 2005.

\bibitem{Hinf_old}
G.~Didinsky, Z.~Pan, and T.~Ba\c{s}ar.
\newblock Parameter identification for uncertain plants using {$H^{\infty}$}
  methods.
\newblock {\em Automatica}, 31:1227--1250, 1995.

\bibitem{Erazo}
K.~Erazo.
\newblock {\em Bayesian Filtering In Nonlinear Structural Systems With
  Applications To Structural Health Monitoring}.
\newblock Graduate College Dissertations and Theses. Paper 513. University of
  Vermont, 2015.

\bibitem{Fridman}
E.~Fridman.
\newblock Observers and initial state recovering for a class of hyperbolic
  systems via {L}yapunov methods.
\newblock {\em Automatica}, 49:2250--2260, 2013.

\bibitem{Haine}
G.~Haine.
\newblock Recovering the observable part of the initial data of an
  infinite-dimensional linear system with skew-adjoint generator.
\newblock {\em Mathematics of Control, Signals, and Systems}, 26(3):435--462,
  2014.

\bibitem{HR12}
G.~Haine and K.~Ramdani.
\newblock Reconstructing initial data using observers: error analysis of the
  semi-discrete and fully discrete approximations.
\newblock {\em Numerische Mathematik}, 120(2):307--343, 2012.

\bibitem{Haraux}
A.~Haraux.
\newblock Une remarque sur la stabilisation de certains syst\`emes du
  deuxi\`eme ordre en temps.
\newblock {\em Portugaliae Mathematica}, 46(3):245--258, 1989.

\bibitem{Imanuvilov}
O.~Imanuvilov and M.~Yamamoto.
\newblock Global {L}ipschitz stability in an inverse hyperbolic problem by
  interior observations.
\newblock {\em Inverse Problems}, 17:717--728, 2001.

\bibitem{Levanony}
D.~Levanony and P.~Caines.
\newblock On persistent excitation for linear systems with stochastic
  coefficients.
\newblock {\em SIAM Journal of Control and Optimization}, 40(3):882--897, 2001.

\bibitem{Liu}
K.~Liu.
\newblock Locally distributed control and damping for the conservative systems.
\newblock {\em SIAM Journal of Control and Optimization}, 35(5):1574--1590,
  1997.

\bibitem{Ljung_EKF}
L.~Ljung.
\newblock Asymptotic behavior of the extended {K}alman filter as parameter
  estimator for linear systems.
\newblock {\em IEEE Transactions on Automatic Control}, 24(1):36--50, 1979.

\bibitem{Luenberger}
D.~Luenberger.
\newblock An introduction to observers.
\newblock {\em IEEE Transactions on Automatic Control}, 16(6):596--602, 1971.

\bibitem{cardiacUKF}
S.~Marchesseau, H.~Delingette, M.~Sermesant, R.~Cabrera-Lozoya, C.~Tobon-Gomez,
  P.~Moireau, R.M.~Figueras i~Ventura, K.~Lekadir, A.~Hernandez, M.~Garreau,
  E.~Donal, C.~Leclercq, S.G. Duckett, K.~Rhode, C.A. Rinaldi, A.F. Frangi,
  R.~Razavi, D.~Chapelle, and N.~Ayache.
\newblock Personalization of a cardiac electromechanical model using reduced
  order unscented {K}alman filtering from regional volumes.
\newblock {\em Medical Image Analysis}, 17(7):816--829, 2013.

\bibitem{Dual_KF}
S.~Mariani and A.~Corigliano.
\newblock Impact induced composite delamination: state and parameter
  identification via joint and dual extended {K}alman filters.
\newblock {\em Computer Methods in Applied Mechanics and Engineering},
  194:5242--5272, 2005.

\bibitem{ROUKF}
P.~Moireau and D.~Chapelle.
\newblock Reduced-order {U}nscented {K}alman {F}iltering with application to
  parameter identification in large-dimensional systems.
\newblock {\em ESAIM: Control, Optimisation and Calculus of Variations},
  17(2):380--405, 2011.

\bibitem{Moireau08}
P.~Moireau, D.~Chapelle, and P.~Le Tallec.
\newblock Joint state and parameter estimation for distributed mechanical
  systems.
\newblock {\em Computer Methods in Applied Mechanics and Engineering},
  197:659--677, 2008.

\bibitem{Ostrowski}
A.~Ostrowski.
\newblock {\em Solution of Equations in Euclidian and Banach Spaces}.
\newblock Academic Press, New York, 1973.

\bibitem{Pazy}
A.~Pazy.
\newblock {\em Semigroups of Linear Operators and Applications to Partial
  Differential Equations}.
\newblock Springer--Verlag, New York, 1983.

\bibitem{Ramdani10}
K.~Ramdani, M.~Tucsnak, and G.~Weiss.
\newblock Recovering the initial state of an infinite-dimensional system using
  observers.
\newblock {\em Automatica}, 46:1616--1625, 2010.

\bibitem{Shimkin}
N.~Shimkin and A.~Feuer.
\newblock Persistency of excitation in continuous-time systems.
\newblock {\em Systems and Control Letters}, 9:225--233, 1987.

\end{thebibliography}

\end{document}